\documentclass{amsart}

\usepackage[latin1]{inputenc}
\usepackage{amsfonts}
\usepackage{amsmath}
\usepackage{amsthm}
\usepackage{amssymb}
\usepackage{latexsym}
\usepackage{enumerate}

\newtheorem{theorem}{Theorem}%[section]
\newtheorem{lemma}[theorem]{Lemma}
\newtheorem{proposition}[theorem]{Proposition}
\newtheorem{corollary}[theorem]{Corollary}

\newtheorem{definition}[theorem]{Definition}

\numberwithin{equation}{section}

\begin{document}

\newcommand{\cc}{\mathfrak{c}}
\newcommand{\N}{\mathbb{N}}
\newcommand{\BB}{\mathbb{B}}
\newcommand{\C}{\mathbb{C}}
\newcommand{\Q}{\mathbb{Q}}
\newcommand{\R}{\mathbb{R}}
\newcommand{\Z}{\mathcal{Z}}
\newcommand{\ZZ}{\mathbb{Z}}
\newcommand{\T}{\mathbb{T}}
\newcommand{\st}{*}
\newcommand{\PP}{\mathbb{P}}
\newcommand{\rin}{\right\rangle}
\newcommand{\SSS}{\mathbb{S}}
\newcommand{\forces}{\Vdash}
\newcommand{\supp}{\text{supp}}
\newcommand{\dom}{\text{dom}}
\newcommand{\osc}{\text{osc}}
\newcommand{\F}{\mathcal{F}}
\newcommand{\A}{\mathcal{A}}
\newcommand{\B}{\mathcal{B}}
\newcommand{\D}{\mathcal{D}}
\newcommand{\I}{\mathcal{I}}
\newcommand{\X}{\mathcal{X}}
\newcommand{\Y}{\mathcal{Y}}
\newcommand{\CC}{\mathcal{C}}
\newcommand{\Car}{\mathfrak{1}}
\newcommand{\sss}{\rm SSW}

\author{Piotr Koszmider}
\address{Institute of Mathematics of the Polish Academy of Sciences,
ul. \'Sniadeckich 8,  00-656 Warszawa, Poland}
\email{\texttt{piotr.math@proton.me}}

\thanks{The author was partially supported by the NCN 
(National Science Centre, Poland) research grant no.\ 2020/37/B/ST1/02613.}

\title[Ramsey-type  properties of the distance in spheres]{ On Ramsey-type   properties
 of the distance in   nonseparable   spheres}

\begin{abstract} 
Given an uncountable subset $\Y$ of a nonseparable Banach space,  is there
an uncountable $\Z\subseteq \Y$ such that the distances between any two distinct points of $\Z$ 
are more or less the same? If an uncountable subset $\Y$ of a nonseparable Banach space does not admit
an uncountable $\Z\subseteq \Y$, where any two points are distant by more than $r>0$, is
it  because $\Y$ is the countable union of sets of diameters not bigger than $r$?

Clearly, these types of questions can be rephrased in the combinatorial language of   partitions 
of pairs of points of a Banach space $\X$  induced by the distance function
$d:[\X]^2\rightarrow \R_+$.   We investigate 
connections between  the set-theoretic phenomena involved  and the geometric properties of uncountable 
subsets  of  nonseparable Banach spaces 
of densities up to $2^\omega$ related to  uncountable $(1+)$-separated sets, equilateral sets or Auerbach systems.

The results include geometric dichotomies for a wide range of classes of Banach spaces, some in {\sf ZFC}, 
some under the assumption
of {\sf OCA}$+${\sf MA}  and some under
a hypothesis on the descriptive complexity of the space as well as constructions (in {\sf ZFC} or under {\sf CH})  
of Banach spaces where the geometry of the 
unit sphere displays anti-Ramsey properties.   This complements classical theorems
for separable spheres and the recent results of H\'ajek, Kania, Russo for densities above $2^\omega$
as well as offers a synthesis of  possible phenomena 
and categorization of examples for uncountable densities up to $2^\omega$
obtained previously by the author and Guzm\'an, Hru\v s\'ak, Ryduchowski and Wark.

\end{abstract}

\maketitle

%\tableofcontents

\section{Introduction and presentation of the main concepts}

The symbol $\R_+$ denotes the set of positive reals and $[M]^2$ the set of all two element subsets of a set $M$.
When $(M, d)$ is a metric space, then we often use $d$ for the function  
$D:[M]^2\rightarrow \R_+$ satisfying $D(\{x, y\})=d(x, y)$ for distinct $x, y\in M$. Moreover we often 
write $d(x, y)$ for $d(\{x, y\})$ when $d:[M]^2\rightarrow \R_+$. 

\begin{definition} Let $M$ be a  set, $d:[M]^2\rightarrow \R_+$  and $r>0$. Then $N\subseteq M$ is said to be
\begin{enumerate}

\item $(r+)$-separated if $d(x, y)>r$ for all distinct $x, y\in N$.
\item $r$-separated if  $d(x, y)\geq r$ for all distinct $x, y\in N$.
\item separated if   it is $r$-separated  for some $r>0$.
\item $r$-equilateral if $d(x,y)=r$ for all distinct $x, y\in N$
\item equilateral if it is $r$-equilateral for some $r>0$.
\item $\varepsilon$-approximately $r$-equilateral if  
$$r-\varepsilon<d(x, y)<r+\varepsilon$$ for all distinct $x, y\in N$.
\item $\varepsilon$-approximately equilateral if it 
is $\varepsilon$-approximately $r$-equilateral for some $r>0$.
\end{enumerate}
As usual the diameter $diam(N)$ is $\sup(\{d(x, y): x, y\in N;\ x\not=y\})$ or $\infty$ if the supremum does not exist. 
A subset $N\subseteq M$ is called
bounded if $diam(N)$ is finite. By the density of a metric space $(M, d)$ we mean the minimal cardinality
of its dense subset, it is denoted by $dens(M)$.
\end{definition}

When talking about  a Banach space $\X$, in particular about its unit sphere,
the notions defined in the above definition refer to $d$ given by
$d(\{x, y\})=d(x, y)=\|x-y\|_\X$ for distinct $x, y$.

Ramsey theory is concerned with combinatorial conditions   on an arbitrary  structure of some kind that yield
the existence of a large regular substructure.  The paradigmatic Ramsey theorem (Theorem \ref{ramsey}) 
asserts the existence of an infinite set $A\subseteq \N$ such that $[A]^2\subseteq K_i$ for some $i\in \{0,1\}$
for an arbitrary partition $[\N]^2=K_0\cup K_1$.  The Ramsey theorem combined with geometric arguments
helped to establish the  existence of an infinite ($1+$)-separated set in the unit sphere of any
infinite dimensional Banach space (Kottman in \cite{kottman}) which was later improved 
to infinite ($1+\varepsilon$)-separated set for some $\varepsilon>0$
using a much more sophisticated Ramsey theory (Elton, Odell in \cite{elton-odell}).
Using the Ramsey theorem one can easily obtain also  infinite $\varepsilon$-approximately equilateral 
subsets  for any $\varepsilon>0$
in any infinite bounded subset of any Banach space (Proposition \ref{omega-ramsey}). However, Ramsey-type  arguments do not
help with infinite equilateral sets: Terenzi showed in \cite{terenzi} that there are infinite dimensional
Banach spaces with no infinite equilateral sets. Thus the geometric and the combinatorial
components both play fundamental roles in this type of results.

If we want to obtain
uncountable regular substructures, we need to consider Ramsey-type results for the uncountable, which
are in general much weaker than the Ramsey theorem. One of such fundamental results concerning
cardinals above the cardinality of the continuum $2^\omega$ is the classical Erd\"os-Rado theorem (Theorem \ref{erdos-rado}).
Together with ingenious
geometric arguments  involving Auerbach systems (for definition see Section 2.1) 
it was used by H\'ajek, Kania 
and Russo to obtain  many deep results, for example,
the existence of uncountable ($1+$)-separated sets in the unit sphere
of any WLD Banach space  of density bigger than $2^\omega$ or in any Banach space
of density bigger than $2^{2^{2^{2^\omega}}}$, also Terenzi noted that direct application
of the Erd\"os-Rado theorem yields uncountable equilateral sets in any Banach space of density bigger than
 $2^{2^\omega}$ (\cite{terenzi}).
The following two basic questions concerning uncountable regular geometric substructures of
nonseparable Banach spaces have been posed in the literature:
\begin{itemize}
 \item[($S$)]  Does the unit sphere
of every nonseparable Banach space admit an uncountable $(1+)$-separated set (see e.g, \cite{hkr})
\item[($E$)]  Does every nonseparable Banach space admit an uncountable equilateral set\footnote{ This is
equivalent to asking if
the unit sphere admits an uncountable $1$-equilateral set, see Proposition \ref{terenzi1}. 
So both problems concern the unit spheres.}, (see e.g., \cite{mer-ck})?
\end{itemize}
Both of the  questions were answered recently in the negative in {\sf ZFC} in \cite{pk-kottman} and
in \cite{pk-hugh}, \cite{pk-kottman} respectively and stronger negative solutions
were obtained consistently in \cite{pk-kamil} and \cite{ad-kottman} .  The spaces 
are of density $2^\omega$. All these  examples, including   the above {\sf ZFC} examples
enjoy rather new metric phenomena in nonseparable Banach spaces, e.g.,  some of them admit 
no uncountable, or even no infinite equilateral sets (\cite{pk-hugh}), 
the unit spheres of some of them are the unions of countably many sets of 
diameters strictly less than $1$ (\cite{ad-kottman}) or some of them  admit
no uncountable Auerbach systems (\cite{pk-kottman}; for definition of Auerbach systems see Section 2.1; 
 first such example was obtained under {\sf CH} in
\cite{hkr}).  At the same time other results showed that the above-mentioned stronger consistent
irregularity properties of Banach spaces of particular types are consistently impossible:
all Banach spaces of the form $C(K)$ admit uncountable
equilateral sets under {\sf MA}+$\neg${\sf CH} (Theorem  5.1 of \cite{equi}), there
are geometric dichotomies for  Johnson-Lindenstrauss spaces and for some of their renormings under {\sf OCA}
(Theorem 2 of \cite{ad-kottman}), the spheres of Banach spaces of Shelah induced by anti-Ramsey colorings
still admit uncountable $(1+)$-separated and equilateral sets under {\sf MA}+$\neg${\sf CH} (Theorem 3 (3) of \cite{pk-kamil}).

The main motivation of the research leading to the results in this paper was to try to understand
the relation of the new geometric phenomena occurring in these examples   with
Ramsey-type combinatorial phenomena at uncountable cardinals up to $2^\omega$.
The issue of Ramsey-type properties
of uncountable cardinals up to $2^\omega$ is, in some sense, much more delicate than the countable
case and than the case above $2^\omega$.  First of all, the purely combinatorial
Ramsey property fails at these cardinals: already Sierpi\'nski constructed
a coloring witnessing this failure  at the first uncountable
cardinal $\omega_1$ or at $2^\omega$ (Theorem \ref{sierpinski1}). Later Todorcevic produced colorings with much stronger 
properties which are particularly useful for constructing structures of size $\omega_1$
with no regular uncountable substructures (Theorem \ref{stevo}).  If one makes
some additional hypothesis like  the continuum hypothesis {\sf CH},
one gets even stronger purely combinatorial (i.e., referring just to colorings without any additional
structure like topology etc.) anti-Ramsey properties (see e.g. \cite{rinot}).

On the other hand there are 
consistent statements of Ramsey-type properties of $2^\omega$ with topological
hypothesis like the Open
Coloring Axiom {\sf OCA} (Theorem \ref{todorcevic}) or Ramsey-type properties of cardinals less than $2^\omega$ with
order-theoretic hypothesis  like Martin's axiom {\sf MA} (Theorem \ref{ma}), or
one can sometimes conclude the existence of uncountable regular substructures under
a descriptive complexity hypothesis of the structure (Theorem \ref{feng}).

The first (of two) main claim  of this paper is that there is
a geometric Ramsey-type property of unit spheres of nonseparable Banach spaces
of densities up to $2^\omega$ which either is implied for wide classes of
Banach spaces by the above mentioned ({\sf OCA}, {\sf MA}) Ramsey-type properties of the uncountable
or holds in {\sf ZFC} for other classes. This property explains the regularity properties of the
{\sf ZFC} examples mentioned above providing the negative answers to Questions (S) and (E).

The regularity properties of the substructures
are revealed through metric dichotomies rather
than the existence of uncountable ($1+$)-separated sets 
and they lead to uncountable approximately equilateral sets in place of uncountable equilateral sets. 
The following is our first (of two) main definition:
\begin{definition}\label{def-dichotomous}
Let $M$ be a set and $d:[M]^2\rightarrow \R_+$. 
We say that $M$ is (almost) dichotomous if for all $0<r\leq diam(M)$ (for all but one\footnote{That this
is a very natural requirement one can see from Proposition \ref{sets-constants}} $0<r\leq diam(M)$) either
$M$ admits an uncountable $(r+)$-separated set  or else
$M$ is the union of countably many sets of diameters not bigger than $r$.
\end{definition}
In the above definition  one sees
 a partition of all pairs of $M$ into two parts: (1) those pairs which are distant (when one
 thinks of
 $d$ as of a metric) by more than $r$ and (2) those pairs
 which are distant by no
 more than $r$. We make a stronger requirement
 on  part (2) than (*) the existence of an uncountable
 set $N\subseteq M$ such that  $[N]^2$ is included in one of the parts. 
 Requiring just (*)  for part (2) would be trivial (in the context of spheres of Banach spaces)
  because  Banach spaces or their unit spheres always admit
 small arcs, that is uncountable sets of small diameters. Half of the focus of this paper is on proving 
 (in {\sf ZFC} or under additional hypotheses)
 that the unit spheres of Banach spaces from many classes 
  are (almost) dichotomous. These are the geometric dichotomies
  mentioned in the abstract and the geometric Ramsey-type regularity properties  mentioned above.
 \begin{theorem}\label{main-zfc} The unit spheres of  the following kinds of Banach 
spaces are dichotomous:
\begin{enumerate}
\item Banach spaces $\X$ which can be isometrically embedded into
$\ell_\infty$ as  analytic\footnote{Recall that a subset of a topological space is analytic
if it is a continuous image of a Borel subset of a Polish space. For more see Section 
2.6 and \cite{kechris}.} subsets of $\R^\N$.  
\item Banach spaces $c_0(\kappa)$,
$\ell_1(\kappa)$, $L_1(\{0,1\}^\kappa)$ for uncountable $\kappa\leq 2^\omega$.
\end{enumerate}
The unit spheres of the following kinds of Banach 
spaces are almost dichotomous:
\begin{enumerate}
  \item[(3)] Banach spaces $\ell_p(\kappa)$ and $L_p(\{0,1\}^\kappa)$
  for uncountable $\kappa\leq 2^\omega$  and $1< p<\infty$.
 \end{enumerate}
\end{theorem}
\begin{proof} This follows from Propositions \ref{feng>dich}, \ref{c0-dich}, \ref{l1}, \ref{Lp-separated},
 \ref{lp-dich}, \ref{Lp}.
\end{proof}
One should add to the above some immediate consequences\footnote{Note that
if the unit sphere $S_\X$ of a Banach space $\X$ admits an uncountable $2$-separated set, then 
it is trivially dichotomous, since the first alternative of Definition \ref{def-dichotomous}
holds for every $0<r\leq diam(S_\X)=2$. The paper \cite{mer-ck}
exactly contains, among other, results on the existence of  uncountable
$2$-equilateral sets in some Banach spaces.} of  known results of Mercourakis and Vassiliadis (\cite{mer-ck}):
The unit spheres of Banach spaces of the form $C(K)$ for $K$ compact, Hausdorff and totally disconnected,
or not hereditarily Lindelof or
not hereditarily separable or carrying a Radon measure of uncountable type are  dichotomous
(see our Proposition \ref{thm-merc} (1)).

\begin{theorem}\label{main-oca} Assume {\sf OCA $+$ MA}. The unit spheres of the following kinds of Banach 
spaces are dichotomous:
\begin{enumerate}
\item Banach spaces whose unit dual ball is separable in the weak$^*$ topology (equivalently
spaces which are isometric to a subspace of $\ell_\infty$).
\item Banach spaces of the form  $C_0(K)$ for a locally compact Hausdorff
space $K$ of weight less than $2^\omega$.
\end{enumerate}
 The unit spheres of the following kinds of Banach 
spaces are almost dichotomous:
\begin{enumerate} 

\item[(3)]  Hilbert generated Banach spaces of  Shelah and their modifications due to
Wark obtained from a coloring $c:[\omega_1]^2\rightarrow\{0,1\}$.
\end{enumerate}
\end{theorem}
\begin{proof} This follows from Propositions \ref{oca>dich}, \ref{C0K-small}, \ref{hilbert-gen-ma}.
\end{proof}
Let us make five comments on the above two theorems.

First note the fundamental role of having dichotomous unit sphere 
 in the context of questions like ($S$): if
  the sphere is dichotomous, and it does not admit an uncountable $(r+)$-separated set 
  then,  it can only be for the canonical reason 
 that it is the union of countably many sets of diameter not bigger than $r$. 
 This is, for example, the case for the spaces of \cite{pk-kottman} and \cite{ad-kottman}
 with no uncountable $(1+)$-separated sets. 
 
 The second comment is about  the role of uncountable Auerbach systems in Banach spaces with
 dichotomous unit spheres:
 
 \begin{theorem}\label{main-auerbach} Suppose that $\X$ is a Banach space whose unit sphere is dichotomous.
 If $\X$ admits an uncountable Auerbach system, then the unit sphere $S_\X$ admits
 an uncountable  $(1+)$-separated set. 
 \end{theorem}
 \begin{proof} Use Proposition \ref{dich>auerbach}.
 \end{proof}
 
 This explains (in the context of Theorem \ref{main-oca} (1)) 
 that it was not a coincidence that our {\sf ZFC} example of \cite{pk-kottman} of a Banach subspace
  of $\ell_\infty$ that did not admit an
 uncountable ($1+$)-separated set also did not admit an uncountable Auerbach system.
 Note that the interaction between  ($1+$)-separated sets and Auerbach system is already very much exploited in \cite{hkr}.
 
 As the third comment  let us mention the impact of having the dichotomous unit sphere not
 only on the topic of Question (S) but also on the topic of Question (E) i.e., on equilateral sets:
 \begin{theorem}\label{main-dich>equi} Suppose that $\X$ is a Banach space whose unit sphere is almost dichotomous.
Then the unit sphere $S_\X$ of $\X$ admits an uncountable  $\varepsilon$-approximately $1$-equilateral sets 
for every $\varepsilon>0$.
 \end{theorem}
 \begin{proof} This follows from Proposition \ref{dich>ramsey}.
 \end{proof}
 
 In the fourth comment we observe that already analytic (in fact Borel) subsets of $\R^\N$ falling
 under Theorem \ref{main-zfc} (1) may provide strong negative answers to questions (S) and (E)
 (Proposition \ref{examples-analytic}).
 So we are additionally justified to  consider the property of having dichotomous unit sphere  as a regularity property.
 
 Finally, these results unify,  complement, generalize and/or extend
some known results. Theorem \ref{main-zfc} (1)
and \ref{main-oca} (2) concern wider classes of spaces than in known results
($C(K)$ spaces for $K$ separable Rosenthal compacta in Proposition 4.7 of Kania and Kochanek's \cite{tt} 
in the first case and $C(K)$ spaces for $K$ compact
 in Theorem 5.1. of \cite{equi} in the second case). These enlargements of the considered classes
of Banach 
spaces are done at the price of allowing the second alternative of Definition \ref{def-dichotomous}.
This is necessary as  there are  {\sf ZFC} examples which satisfy the second part of the dichotomy
and not the first (Propositions
\ref{examples-analytic} and \ref{C0K-strange}).  On the other hand the dichotomy
of Theorem \ref{main-oca} generalizes the dichotomies for
Johnson-Lindenstrauss spaces  and for some of their renormings in \cite{ad-kottman}.
Parts of the content of Theorems \ref{main-zfc} (2), (3) include and  generalize 
the argument  of Erd\"os and Preiss (\cite{erdos-p}) that the unit ball of $\ell_2(2^\omega)$
is the countable union of sets of diameters less than $\sqrt2+\varepsilon$ for every $\varepsilon>0$
and a result of Preiss and R\"odl (\cite{preiss-r}) that it is not  the countable union of sets of diameters not
bigger than $\sqrt2$.
Similarly Theorem \ref{main-oca} (3) includes Theorem 3 (3) of \cite{pk-kamil} asserting the existence
of uncountable $\sqrt 2$-equilateral (and so $\sqrt 2$-separated) set in the unit ball of the Banach spaces
of these results.

 The second main claim of this paper is that the strong failure of
 the purely combinatorial Ramsey property yields the existence of 
 strong anti-Ramsey-type  unit  spheres of Banach spaces. The following
 is our second main definition. 
 \begin{definition}\label{def-hyperlateral}
Let $M$ be a set and $d:[M]^2\rightarrow \R_+$.
We say that $M$ is hyperlateral  if for every uncountable separated $N\subseteq M$
there is $\varepsilon>0$ such that $N$ does not contain an
uncountable $\varepsilon$-approximately equilateral subset.
\end{definition}
In the above definition one sees a collection of partitions $\mathcal P_\varepsilon$ for
$\varepsilon>0$ of pairs of $M$ each into countably many parts such that
if two pairs $\{m, m'\}$ and $\{n, n'\}$ are in the same part, then 
$|d(m, m')-d(n, n')|<2\varepsilon$.  $M$ is hyperlateral if no uncountable separated
$N\subseteq M$ admits uncountable monochromatic sets for all partitions $\mathcal P_\varepsilon$,
i.e., such $N'\subseteq N$ that $[N']^2$ is included in one part of the partition.
Here we need to consider infinitely many partitions because in the unit spheres of nonseparable Banach spaces
 we always have uncountably many points close to each other.
 
It turns out that for $M$ being the unit sphere of
a Banach space $\X$ to be hyperlateral is equivalent to the entire space $\X$ being hyperlateral
and equivalent to the nonexistence of uncountable $\varepsilon$-approximately $1$-equilateral sets
in the unit sphere and to the existence for every $s>0$ of an $\varepsilon>0$ such that
all uncountable $s$-separated sets do not admit an uncountable $\varepsilon$-approximately equilateral
sets (Proposition \ref{sphere-no}).  It should be also clear that being hyperlateral corresponds to
anti-Ramsey properties of Sierpi\'nski's or Todorcevic's colorings (Theorems \ref{sierpinski1}, \ref{stevo})
as any uncountable separated set hits many colors-distances, where many means with fixed nonzero
oscillation (fixed for a given separated set where subsets are considered).
Our results on hyperlateral Banach spaces consist of proving (in some cases, just noting)  that  some spaces 
considered in the papers \cite{ad-kottman, equi, pk-kamil} are hyperlateral and exploring their additional properties
in this context.

 \begin{theorem}\label{main-ch}{\rm (\cite{ad-kottman}, \cite{equi}, \cite{pk-kamil})}   
Assume {\sf CH}. There are nonseparable Banach spaces 
of density $\omega_1$ whose unit spheres are  hyperlateral
and which are members of the following classes:
\begin{enumerate}
\item  Banach subspaces of $\ell_\infty$.
\item  Banach spaces of the form $C(K)$ for $K$ compact, Hausdorff and separable.
\item  Hilbert generated Banach spaces of  Shelah and their modifications due to
Wark obtained from a coloring $c:[\omega_1]^2\rightarrow\{0,1\}$.
\end{enumerate}
\end{theorem}
\begin{proof} This follows from Propositions \ref{ad-hyper}, \ref{hyper-ck}, \ref{hilbert-gen-ch}. 
\end{proof}
The properties corresponding to being
hyperlateral were first obtained in \cite{ad-kottman} by Guzm\'an,  Hru\v s\'ak and the author
and also presented in \cite{pk-kamil} by Ryduchowski and the author, but we later
realized that the example of \cite{equi} is hyperlateral as well which is proved
in Proposition \ref{hyper-ck} together with deriving the construction from (an axiom
weaker than) {\sf CH} in place of a forcing construction presented in \cite{equi}.

 Despite the above two groups of results showing the impact of Ramsey-type and anti-Ramsey
 set-theoretic phenomena on the  distance function in unit spheres of nonseparable
 Banach spaces we do not settle
 completely the questions concerning the extent of this impact. We do not know the answers to the
 following version of Question (\textit{S}) and a reversed version of Question (\textit{E}):
 
 \begin{itemize}
\item[($S'$)] Is it consistent that all Banach spaces of density $\omega_1$ have
almost dichotomous unit spheres?
\item[($E'$)] Is there in {\sf ZFC} a nonseparable Banach space which is  hyperlateral?
\end{itemize}
Note that $\omega_1$ cannot be replaced in ($S'$)  by $2^\omega$ because we obtain the following
 \begin{theorem}\label{main-zfc-c}
  There are Banach spaces of density $2^\omega$
of the form $C_0(K)$ for $K$ locally compact and Hausdorff whose unit sphere $S_{C_0(K)}$
is not   almost dichotomous.
 \end{theorem}
 \begin{proof} This follows from Proposition \ref{C0K-counter}.
 \end{proof}
However, $\omega_1$ in Question $({S'})$ can be considered quite natural if we remember that the original
 Question $(S)$ concerns the existence of uncountable subsets, and so of cardinality $\omega_1$.
 Also in the spheres of spaces like $c_0(\kappa)$ one cannot hope
 for $(1+)$-separated sets of cardinalities bigger than $\omega_1$ regardless of $\kappa$
 (Theorem A (iii) of \cite{hkr}).
 The positive answer to Question $({S'})$ would imply the negative answer to Question $({E'})$ by Theorem \ref{main-dich>equi}
 and Proposition \ref{sphere-no}.
 
As Sierpi\'nski's or Todorcevic's colorings exist in {\sf ZFC} one could expect that
the answer to Question $({E'})$ is positive and so the answer to question $({S'})$ is negative.
In fact the failure of the Ramsey property as in Definition \ref{def-hyperlateral}
 but on a single uncountable subset of the unit sphere of a Banach space can be easily obtained in {\sf ZFC} from
 known anti-Ramsey colorings of uncountable squares.
 This is because  any combinatorial coloring can be ``translated'' to the unit sphere of a Banach space
(Lemmas \ref{color-dist} and \ref{kuratowski})  and so one can inject strong anti-Ramsey colorings of Sierpi\'nski or Todorcevic  
witnessing negative square-bracket partition relations 
into the spheres of Banach spaces (Proposition  \ref{antisubbanach}  (1)).
 The rich algebraic and topological structure in a Banach space
seem to be here an obstruction to transferring the property of the coloring to all uncountable subsets of the sphere
even if the   Banach space is generated by
 subset of the unit sphere which originates from an anti-Ramsey coloring and is hyperlateral (see 
Proposition \ref{hilbert-gen-ma}). However,  similar obstructions were eventually somehow overcame in many
contexts where anti-Ramsey structures
were constructed in {\sf ZFC},  for  example for groups (\cite{extra, poor})
 or in functional analysis (\cite{antonio, ss}).

Combining the above results sometimes with their more specific versions in the 
further text we may obtain many corollaries. Let us mention some examples of them:
 
 \begin{corollary} The following statements are independent from {\sf ZFC}:
 \begin{enumerate}
  \item Every uncountable set in $\ell_\infty$ admits an uncountable
 $\varepsilon$-approximately equilateral subset for every $\varepsilon>0$.
 \item There is a nonmetrizable locally compact Hausdorff space of weight $\omega_1$
 such that the unit sphere of the Banach space $C_0(K)$ is not the countable union of 
 sets of diameters  $r<2$ and it does not admit an uncountable
 $(1+\varepsilon)$-separated set for any $\varepsilon>0$.
 \end{enumerate}
 \end{corollary}
 \begin{proof} For (1) use Proposition \ref{oca>dich} together with Proposition
 \ref{dich>ramsey} for one consistency and Theorem \ref{main-ch} (1) for the other.
 For (2) one of the consistencies follows from  Theorem \ref{main-oca} (2). The other
 consistency is a consequence of 
 Theorem \ref{main-ch} (2). Indeed, for a compact $K$ the existence
 of an uncountable $2$-equilateral set in the unit sphere is equivalent to the 
 existence of $\varepsilon>0$ and an  uncountable $(1+\varepsilon)$-separated  subset in the unit sphere
 by Theorem 2.6 of \cite{mer-ck}. So the space $C(K)$ with the hyperlateral
 unit sphere of Theorem \ref{main-ch} (2) cannot admit any of the above sets
 as the spheres with uncountable $2$-equilateral sets are clearly dichotomous.
 
 \end{proof}
 
 \begin{corollary} $ $
 \begin{enumerate}
 \item If a nonseparable  subspace of $\ell_\infty$ has an uncountable analytic Auerbach system, then its 
unit sphere admits an uncountable $(1+)$-separated set.
\item  If a nonseparable subspace of $\ell_\infty$ is analytic, then its 
unit sphere admits an uncountable $\varepsilon$-approximately $1$-equilateral sets for every $\varepsilon>0$.
 \end{enumerate}
 \end{corollary}
 \begin{proof} 
 For (1) use Proposition \ref{feng-auerbach}.
 Item (2) follows from Theorem \ref{main-zfc} (1) and Proposition \ref{dich>ramsey+} and
 Proposition \ref{sphere-no} (5).
 \end{proof}

The structure of the paper is as follows. In the next section we gather well known results needed in the following sections
and prove some general results relevant to the rest of the paper. In the first subsection 2.1
we review our terminology and notation. In Section 2.3 we introduce
coefficients which say for which $r$ which part of the dichotomies $D_r$ hold. In fact all results in
the paper are supplied with this information whenever possible. We skipped this in 
the results presented in the introduction  not to clog it.
In sections 3, 4, 5 we present the parts of  Theorems \ref{main-zfc}, \ref{main-oca},
\ref{main-ch} pertaining to Banach spaces with weakly$^*$ separable dual balls, 
some WCG Banach spaces and
Banach spaces of continuous functions respectively.
Besides general theorems the paper contains one new construction in Proposition
\ref{C0K-counter}. However, many constructions already
 presented in the literature have their properties
substantially refined here. The fifth section also contains 
the analysis of the unit spheres of $C_0(K)$ spaces for $K$ locally compact
in the context of uncountable equilateral and $(1+)$-separated sets
which has not been done much in the literature so far and may be of
independent interest. The last section contains concluding remarks.

We would like to thank W. Marciszewski for two reasons. First, for 
motivating the author to prepare a survey seminar presentation concerning
the results of papers \cite{equi, pk-kottman, pk-hugh, ad-kottman, pk-kamil}. 
This lead to some observations which started the research which resulted in this paper. 
And secondly, for convincing the author  that the spaces
in Proposition \ref{examples-analytic} are not only analytic but in fact Borel.
In particular the proof of Lemma \ref{span-Borel} was provided by W. Marciszewski to the author.
We would like also  to thank the anonymous referee for very careful reading of the initial
version of this paper and indicating many inaccuracies.

\section{Preliminaries}

\subsection{Notation and terminology} 

Undefined set-theoretic notions and symbols could be found in \cite{jech} and
additional descriptive set-theoretic notions like related to Polish spaces and analytic sets can be found in \cite{kechris}.
Undefined Banach space theoretic notions and symbols could be found in \cite{fabianetal} . 

In particular $[X]^2$ denotes the family of all two element subsets of a set $X$. 
If $c:[X]^2\rightarrow Y$, then $X'\subseteq X$ is said to be $y$-monochromatic 
or just monochromatic if $c[[X']^2]=\{y\}$.  By $\kappa^+$ we mean
the smallest cardinal bigger than cardinal $\kappa$. The symbol $\omega_1$ denotes the first
uncountable cardinal and $2^\omega$ the cardinality of $\R$. If $A$ is a  subset of
a set $X$ clear from the context by $1_A$ we denote the characteristic function of $A$
whose domain is $X$.
 By
the continuum hypothesis {\sf CH} we mean the statement that $2^\omega=\omega_1$.
{\sf ZFC} denotes the Zermelo-Fraenkel set theory with the axiom of choice.
The Open Coloring Axiom {\sf OCA} is defined in Theorem \ref{todorcevic}.
The only consequence of Martin's axiom {\sf MA} that we will be using is Theorem \ref{ma}. 
For more details see \cite{jech}.

If $\X$ is a Banach space, then $S_\X$ denotes its unit sphere and $B_{\X^*}$ the
unit ball in the dual space $\X^*$.  By a Banach space of the form $C(K)$ ($C_0(K)$)
we mean  the Banach space of all real valued continuous functions on $K$ (of all continuous
functions where sets $\{x\in K: |f(x)|\leq\varepsilon\}$ are compact for all $\varepsilon>0$)
with the supremum norm denoted by $\|\ \|_\infty$, where $K$ is a compact Hausdorff space ($K$ is a locally 
compact Hausdorff space).
Here $K$ does not need to be specified if we talk about Banach spaces of the form $C(K)$ ($C_0(K)$).
The symbol $\ell_\infty$ stands for the Banach spaces of all bounded elements of $\R^\N$ 
with the supremum norm denoted by $\|\ \|_\infty$.
If $\kappa$ is a cardinal by the support of an element $x\in \R^\kappa$ we mean
$\{\alpha<\kappa: x(\alpha)\not=0\}$, it is denoted by $supp(x)$. By 
 $c_{00}(\kappa)$ we mean the subset of $\R^\kappa$ of all elements with finite supports.
 The Banach space $c_0(\kappa)$ consisting of all elements $x$
 of $\R^\kappa$ where $\{\alpha<\kappa: |x(\alpha)|\geq\varepsilon\}$ is finite
 for every $\varepsilon>0$ is considered as usual with the supremum norm denoted by $\|\ \|_\infty$.
 Similarly the Banach spaces $\ell_p(\kappa)$  for $1\leq p<\infty$ consisting of all elements $x$
 of $\R^\kappa$ where 
 $$\|x\|_p=\sqrt[p]{\sup(\{\Sigma_{\alpha\in F}|x(\alpha)|^p: F\subseteq\kappa, \ \hbox{is finite}\})}$$ is finite
 is considered as usual with the above $\|\ \|_p$ norm. 
 If $\X$ is a Banach space, by an Auerbach system we mean 
 $\{(x_i, x^*_i): i\in I\}\subseteq S_\X\times S_{\X^*}$ such that $x^*_i(x_j)=0$ unless
  $i=j$ and then $x^*_i(x_j)=1$. For more on Auerbach and biorthogonal systems see \cite{hajek}.
 
 \subsection{Classes of Banach spaces}
 Let us make a few comments on classes of nonseparable Banach spaces  of density not bigger than $2^\omega$
 that we consider in this paper.  They fall into three  subclasses: subspaces of $\ell_\infty$, 
 WCG (weakly compactly generated) Banach spaces and Banach spaces
 of continuous functions on nonmetrizable locally compact
 Hausdorff spaces.  The first subclass contains
 all Banach spaces of the form $C_0(K)$ for $K$ separable nonmetrizable and $\ell_1(\kappa)$
 for $\omega_1<\kappa\leq2^\omega$ while the second contains all
 spaces $\ell_p(\kappa)$, $c_0(\kappa)$
 for $\omega_1<\kappa\leq2^\omega$ and $1<p<\infty$. The first two classes are disjoint.
 To see how these classes differ it is perhaps better to consider  the
  class of WLD (weakly Lindel\"of determined) Banach spaces which contains all WCG
  Banach spaces and is characterized as the class of Banach
spaces $\X$ which admit a linearly dense set $D\subseteq \X$ (i.e., whose
span is dense) such that $\{d\in D: x^*(d)\not=0\}$ is
countable for each $x^*\in \X^*$ (\cite{wld}). On the other hand
subspaces of $\ell_\infty$ are exactly the Banach spaces
whose dual ball is weak$^*$ separable, in particular
they admit a countable $\{x_n^*: n\in \N\}\subseteq\X^*$
such that for any $x\in \X$ the condition  $x_n^*(x)=0$ for all $n\in \N$ implies that $x=0$.
In the topological language one class has separable nonmetrizable dual balls
while the other Corson compact dual balls (and so with each separable subspace metrizable).
The class of spaces of continuous functions intersects first two classes
but there are its elements which do not belong to any of 
them like the space of Proposition \ref{C0K-counter}.

\subsection{Numerical constants related to the metric dichotomies}

\begin{definition}\label{def-sets-constants}
Let $M$ be a set and $d:[M]^2\rightarrow \R_+$ be bounded. We introduce the following
$$\Sigma(M, d)=\{r\in \R_+\cup\{0\}:  M=\bigcup_{n\in \N} M_n, \ diam(M_n)\leq r \leq diam(M)\
 \ \hbox{\rm for all}\ n\in \N\}$$
 $$\sigma(M, d)=\inf(\Sigma(M, d))$$
 $${\mathsf K}^+(M, d)=\{r\in \R_+\cup\{0\}:  \exists N\subseteq M  \ N \ 
 \hbox{\rm is uncountable and $(r+)$-separated}\}$$
 $${\mathsf k}^+(M, d)=\sup({\mathsf K}^+(M, d)).$$
${\mathsf k}^+(M, d)$ will be  called the uncountable Kottman's $(+)$-constant of $M$.
 If $d$ is clear from the context we use $\Sigma(M)$, $\sigma(M)$, ${\mathsf K}^+(M)$, 
 ${\mathsf k}^+(M)$
instead of $\Sigma(M, d)$, $\sigma(M, d )$, ${\mathsf K}^+(M, d)$, ${\mathsf k}^+(M, d)$.
\end{definition}

If we replace uncountable $(r+)$-separated sets by infinite $r$-separated sets in the definition
of the uncountable Kottman's $(+)$-constant we obtain a well investigated Kottman's constant,
see for example \cite{castillo}. We immediately obtain the following

\begin{proposition}\label{sets-constants}
Suppose that $M$ is a set and $d:[M]^2\rightarrow\R_+$ is bounded. Then
\begin{enumerate}
\item ${\mathsf K}^+(M)$ and $\Sigma(M)$ are disjoint subintervals of $[0,diam(M)]$
with $$0\leq\sup({\mathsf K}^+(M))={\mathsf k}^+(M)\leq\sigma(M)=\inf(\Sigma(M))\leq diam(M).$$
\item $(M, d)$ is dichotomous if and only if ${\mathsf K}^+(M)\cup\Sigma(M)=[0,diam(M)]$.
\item $(M, d)$ is almost dichotomous if and only if ${\mathsf k}^+(M)=\sigma(M)$.
\end{enumerate}
\end{proposition}

The above proposition explains why being almost dichotomous as in Definition \ref{def-dichotomous}
is a natural condition.

\begin{proposition}\label{riesz} Let $\kappa$
be an infinite cardinal and $\X$ be an infinite dimensional  Banach space of density $\kappa$
and let $\varepsilon>0$.
 $S_\X$ admits a $(1-\varepsilon)$-separated set of cardinality $\kappa$ and is of diameter $2$.
 In particular, if $\X$ is nonseparable,  then 
 \begin{enumerate}
 \item $[0,1)\subseteq{\mathsf K}^+(S_\X)$, ${\mathsf k}^+(S_\X)\geq1$,
 \item $[0, diam(S_\X)]\setminus ({\mathsf K}^+(S_\X)\cup \Sigma(S_\X))$ is 
 a subinterval of $[1,2)$,
 \item  $2\in \Sigma(S_\X)$, $\sigma(S_\X)\leq2$.
 \end{enumerate}
\end{proposition}
\begin{proof} Construct the set from the first part of the proposition by transfinite induction.
Having constructed its first $\lambda$ elements for some $\lambda<\kappa$
consider the closed subspace $\Y$ of $\X$ generated by the constructed elements.
To find the next element use the Riesz lemma which says that
for every Banach space $\X$ and its proper closed subspace $\Y\subseteq \X$, for every $\varepsilon>0$
there is $x\in S_\X$ such that $\|x-y\|>1-\varepsilon$ for every $y\in \Y$.
The diameter $2$ follows from the triangle inequality. The remaining items follow
from the first part and Proposition \ref{sets-constants}.
\end{proof}

\subsection{Dichotomous spheres, approximately equilateral sets and Auerbach systems}

\begin{proposition}\label{terenzi1} {\rm (Terenzi \cite{terenzi})}
A Banach space $\X$ admits an uncountable equilateral set if and only if
its unit sphere $S_{\X}$ admits $1$-equilateral set.
\end{proposition}
\begin{proof}
Suppose that $\Y\subseteq\X$ is $r$-equilateral for some $r>0$.
Consider $\{(y/r)-(y_0/r): y\in \Y\setminus\{y_0\}\}$ for some $y_0\in \Y$.
\end{proof}

Recall that if  $\X$ is a Banach space, by Auerbach system we mean 
 $\{(x_i, x^*_i): i\in I\}\subseteq S_\X\times S_{\X^*}$ such that $x^*_i(x_j)=0$ unless $i=j$ and then $x^*_i(x_j)=1$. 

\begin{proposition}\label{dich>auerbach} Let $\X$ be a Banach which admits an uncountable Auerbach system
$\{(x_i, x^*_i): i\in I\}$  such that $\{x_i: i\in I\}\subseteq \Y$, where $\Y \subseteq S_\X$
is  dichotomous
{\rm (}e.g., $\Y=S_\X$ or $\Y=\{x_i: i\in I\}${\rm )}. Then $S_\X$ admits an uncountable
$(1+)$-separated set.  
\end{proposition}
\begin{proof}
Suppose that there is no uncountable $(1+)$-separated set in $S_\X$ and let us arrive at a contradiction.
Under  the hypothesis $\Y$ and so $\{x_i: i\in I\}$ is the countable union of sets of diameters not
bigger than $1$.  By renumerating a subsystem of the Auerbach  system we conclude that   there is an uncountable
Auerbach system $\{x_\alpha: \alpha<\omega_1\}$ which has diameter not bigger than $1$
and so it is an uncountable  $1$-equilateral set since $1=x^*_\alpha(x_\alpha-x_\beta)\leq\|x_\alpha-x_\beta\|$.
 The rest of the proof follows 
the proof of Proposition 3.4 of \cite{hkr}.
Let $(c_i)_{i\in\N}$ have all positive terms and $\Sigma_{i=0}^\infty c_i=1$. For
each $\omega\leq\alpha<\omega_1$ let $e_\alpha:\N\rightarrow \alpha$ be a bijection.
Consider
$$y_\alpha= x_\alpha-\Sigma_{i=0}^\infty c_i x_{e_\alpha(i)}.$$
Note that $\|y_\alpha\|\leq \Sigma_{i=0}^\infty c_i \|x_\alpha-x_{e_\alpha(i)}\|=1$ and
$1=x^*_\alpha(y_\alpha)\leq\|y_\alpha\|$, so $\|y_\alpha\|=1$.
On the other hand for every $\beta<\alpha<\omega_1$ for $i\in\N$ such that $e_\alpha(i)=\beta$ we have
$$|-c_i-1|=|x^*_\beta(y_\alpha-y_\beta)|\leq\|y_\alpha-y_\beta\|.$$
So $\{y_\alpha: \alpha<\omega_1\}$ is $(1+)$-separated. A contradiction.
\end{proof}

\begin{proposition}\label{dich>ramsey} Let $d: [M]^2\rightarrow \R_+$. 
If $M$ is uncountable, bounded and  almost dichotomous with  ${\mathsf k}^+(M)=\sigma(M)=s>0$,
then for every $\varepsilon>0$ the set $M$ contains an uncountable 
subset $N$ which is   $\varepsilon$-approximately $s$-equilateral.  
\end{proposition}
\begin{proof}
By Proposition \ref{sets-constants} (3) the fact that $M$ is almost dichotomous implies that ${\mathsf k}^+(M)=\sigma(M)$.
Let $\varepsilon>0$ and let  $ M$ be of diameter $R$ and uncountable.
Since $s={\mathsf k}^+(M)>0$, 
there is an uncountable $(s-\varepsilon+)$-separated subset $N'$ of $M$.
Since $s=\sigma(M)<\infty$ as $M$ is bounded,
$M$ is the union of countably many
  sets $M_n$ for $n\in \N$ of diameter not bigger than
 $(s+\varepsilon)$.   So there is $n\in \N$ such that  $N=M_n\cap N'$ is uncountable.
 So $s-\varepsilon< d(y, y')\leq s+\varepsilon$  for all $y, y'\in N$, as required.
\end{proof}

\begin{proposition}\label{sphere-no}  Suppose that $\X$ is a nonseparable Banach space.  The following are equivalent
\begin{enumerate}
\item $\X$ is hyperlateral.
\item For every $s>0$ there is $\varepsilon>0$ such that no uncountable $s$-separated
$\Y\subseteq \X$ contains an uncountable $\varepsilon$-approximately equilateral set.
\item $S_\X$ is hyperlateral.
\item For every $0<r\leq 2$ there is $\varepsilon>0$   such that there is no 
uncountable $\Y\subseteq S_\X$ such that 
$$r-\varepsilon<\|y-y'\|<r+\varepsilon$$
for every distinct $y, y'\in\Y$.
\item There is $\varepsilon>0$ such that there is no 
uncountable $\Y\subseteq S_\X$ such that 
$$1-\varepsilon<\|y-y'\|<1+\varepsilon$$
for every distinct $y, y'\in\Y$.

\end{enumerate}
\end{proposition} 

\begin{proof} 

To prove that (1) implies (2) suppose that (2) fails and let us contradict
(1). If (2) fails, then  there is $s>0$ and  for every  natural $n>1$ there is an $s$-separated 
$\Y_{n}\subseteq \X$ of cardinality $\omega_1$ which is $(1/n)$-approximately equilateral.

 We will construct $\Z\subseteq \bigcup_{n\in \N}\Y_n$ of cardinality $\omega_1$ which
is $s/3$-separated and contains uncountable $(1/n)$-approximately equilateral set for arbitrary  $n\in\N$.
This will contradict (1). Let $(A_n: n\in \N)$ be a partition of $\omega_1$ into sets of cardinality $\omega_1$.
Construct by induction on $\xi<\omega_1$ the elements $x_\xi\in \Y_{n}$ for $\xi\in A_n$
such that $\|x_\eta-x_\xi\|>s/3$ for all $\eta<\xi$.  This can be done since if there
were two elements $x$ of $\Y_{n}$ such that $\|x_\eta-x\|\leq s/3$ for fixed $\eta<\xi$, we would contradict
the hypothesis that $\Y_{n}$ is $s$-separated.   So 
$\{x_\xi: \xi<\omega_1\}\subseteq \X$ is separated  and contains 
uncountably many terms of $\Y_{n}$  for all $n\in\N$ which is $(1/n)$-approximately equilateral.  
This contradicts (2) as required.
The reverse implication is clear, so (1) and (2) are equivalent.

The same argument implies that (3) and (4) are equivalent.
The implication (2)$\Rightarrow$(3) is clear.  
 The implication (4) $\Rightarrow$ (5) is clear. 

To prove that (5) implies (1)  we will suppose that there is  an uncountable  separated  $\Y\subseteq \X$
 which contains uncountable $\varepsilon$-approximately equilateral subsets for every $\varepsilon>0$ and 
 for every $\varepsilon>0$ we will construct
sets $\Z\subseteq S_\X$ of cardinality $\omega_1$ such that $1-\varepsilon<\|y-y'\|<1+\varepsilon$
for every distinct $y, y'\in\Z$.  

Let $r>0$  be such that $r\leq \|y-y'\|$
for every distinct $y, y'\in\Y$.  Using the failure of (1) for $\Y$ find an uncountable $\Y_\varepsilon\subseteq \Y$
and $s_\varepsilon\geq r$  such that  
$s_\varepsilon-\varepsilon<\|y-y'\|<s_\varepsilon+\varepsilon$ for every distinct $y, y'\in\Y_\varepsilon$.
So we have $1-\varepsilon/s_\varepsilon<\|y/s_\varepsilon-y'/s_\varepsilon\|<1+\varepsilon/s_\varepsilon$
and consequently $1-\varepsilon/r<\|z-z'\|<1+\varepsilon/r$
for all distinct $z, z'\in \Z_\varepsilon=\{y/s_\varepsilon: y\in \Y_\varepsilon\}$.

Note that then for distinct $z_\varepsilon, z\in \Z_\varepsilon$ we have 
$$\Big|{1\over{\|z_\varepsilon-z\|}}-{1}\Big|=\Big|{{1-\|z_\varepsilon-z\|}\over{\|z_\varepsilon-z\|}}
\Big|\leq {\varepsilon\over{r(1-\varepsilon/r)}}.\leqno(*)$$
Fix $z_\varepsilon\in \Z_\varepsilon$ and consider 
$$\mathcal U_\varepsilon=\Big\{{{z_\varepsilon-z}\over{\|z_\varepsilon-z\|}}: 
z\in \Z_\varepsilon\setminus \{z_\varepsilon\}\Big\}\subseteq S_\X.$$
Note that for $z, z'\in \Z_\varepsilon\setminus \{z_\varepsilon\}$ we have 
$$\Big\|  {z_\varepsilon-z\over{\|z_\varepsilon-z\|}} -  {z_\varepsilon-z'\over{\|z_\varepsilon-z'\|}   }  \Big\|
=\Big\|   {(z'-z) } +(z_\varepsilon-z)\Big({1\over{\|z_\varepsilon-z\|}}-{1}\Big)-(z_\varepsilon-z')
\Big({1\over{\|z_\varepsilon-z'\|}}-{1}\Big)        \Big\|.$$
So by (*) the distances between
any two distinct members of $\mathcal U_\varepsilon$ are
in between $1-\varepsilon/r-2{(1+\varepsilon/r)\varepsilon\over r(1-\varepsilon/r)}$ and 
$1+\varepsilon/r+2{(1+\varepsilon/r)\varepsilon\over r(1-\varepsilon/r)}$.
Since $r>0$ is fixed and $\varepsilon>0$ can be arbitrarily small this completes the proof.

\end{proof}

\begin{proposition}\label{dich>ramsey+} Suppose that $\X$ is a nonseparable Banach space
whose unit sphere $\X$ is almost dichotomous. Then $\X$ and $S_\X$ are not hyperlateral.
\end{proposition}
\begin{proof} By Proposition \ref{riesz}  uncountable Kottman's $(+)$-constant  of $S_\X$ satisfies 
$s={\mathsf k}^+(S_\X)\geq1$.
So by Proposition \ref{dich>ramsey} for every $\varepsilon>0$ the sphere $S_\X$ contains uncountable
 $\varepsilon$-approximately $s$-equilateral sets which shows that
$S_\X$ is not hyperlateral by Proposition \ref{sphere-no} (4).
\end{proof}

\subsection{General colorings, metric spaces and unit spheres of Banach spaces}

The following lemma allows us to transfer colorings of pairs with two colors into metrics:

\begin{lemma}\label{color-dist} Suppose that $\kappa$ is a cardinal and  $c: [\kappa]^2\rightarrow[0,1]$.
Then $d_c: \kappa\times\kappa\rightarrow[0,2]$ defined by
\[
 d_c(\alpha, \beta) =
  \begin{cases}
    0 & \text{if $\alpha=\beta$,} \\
   1+c(\{\alpha, \beta\})  & \text{otherwise}.
  \end{cases}
\]
is a metric bounded by $2$. 
\end{lemma}
\begin{proof} Let $d=d_c$. Consider $\alpha, \beta, \gamma$ in $\kappa$ and the triangle inequality
$d(\alpha, \beta)\leq d(\alpha, \gamma)+d(\gamma, \beta)$.  We may assume that $\alpha\not=\beta$.
If $\gamma\in \{\alpha, \beta\}$, then the right hand side also contains $d(\alpha, \beta)$ so the inequality is satisfied.
If $\gamma\not\in \{\alpha, \beta\}$, then both terms of the right hand side are at least one, so we are done as well
since the left hand side is bounded by $2$. 
\end{proof}

\begin{lemma}\label{kuratowski} Suppose that $(M, d)$ is a metric space 
of diameter not bigger than $2$. Then $M$ can be isometrically embedded into 
the unit sphere in a  Banach space of continuous functions on a compact Hausdorff space.
\end{lemma}
\begin{proof}
Let $(M, d)$ be a metric space where the metric is bounded by $2$. Add one new point $x_0$ to
$M$ declaring the distance from $x_0$ to any point in $M$ to be $1$. We claim that 
$M\cup\{x_0\}$ is metric.
The triangle inequality is satisfied because
$d(x_0,y)=1\leq  1+d(y, y')= d(x_0, y')+d(y', y)$ and
$d(y, y')\leq 2=d(y, x_0)+d(x_0, y')$ for any $y, y'\in M$. So indeed
$M\cup\{x_0\}$ is metric and bounded by $2$. By
the Kuratowski embedding theorem $\Psi: M\rightarrow \ell_\infty(M)$ defined
by $\Psi(y)(x)=d(x, y)$  for any $x, y\in M$ is an isometric embedding.
 The space $\ell_\infty(M)$ is 
isometric to $C(K)$ for $K$ being $\beta M$.
 Then $x_0$ is sent to some $\Psi(x_0)\in C(K)$ and all
remaining points are sent to points distant by $1$ to $\Psi(x_0)$. So using (isometric) translation $\Phi$ we can
move the copy of $M\cup \{x_0\}$ so that $x_0$ is moved to $\Phi(\Psi(x_0))=0$. 
Then $\Phi[\Psi[M]]$ must be a subset of the unit sphere in $C(K)$ isometric to $M$.
\end{proof}

\subsection{Ramsey-type properties of the countable and the uncountable and some  consequences}

In this subsection we recall   the Ramsey property of the countable  and we discuss the Ramsey-type properties of the uncountable
with their
direct consequences relevant for this paper.

\begin{theorem}[Ramsey]\label{ramsey} 
Suppose that $X$ is an infinite  set and $c:[X]^2\rightarrow\{0,1\}$.
There is an infinite $Y\subseteq X$ such that $c[[Y]^2]$ is a singleton.
\end{theorem}

Now note an elementary impact of the Ramsey theorem on the existence of
 infinite $\varepsilon$-approximately equilateral sets:

\begin{proposition}\label{omega-ramsey} Suppose that  $M$ is infinite,  $d: [M]^2\rightarrow \R_+$
and $M$ is bounded and separated with respect to $d$.
Then for every $\varepsilon>0$ the set $M$ contains an infinite
$\varepsilon$-approximately equilateral set.
\end{proposition}
\begin{proof}
Let $R$ be the diameter of $M$.  Let $I_1, \dots I_n$ for some $n\in \N$ be intervals
of diameters $\varepsilon$ which cover $[0,R]$. Consider $c:[M]^2\rightarrow\{1, \dots, n\}$
defined by $c(x, x')=k$ if $k$ is the smallest such that $d(x, x')\in I_k$.  By the Ramsey theorem
there is an infinite monochromatic set for $c$. It is $\varepsilon$-approximately equilateral set, as required.
\end{proof}

Proposition \ref{omega-ramsey} can be traced back to Matousek (\cite{matousek}, cf. 
\cite{geschke-dist}).  
In \cite{kottman} Kottman has used Ramsey theorem to show that the unit sphere in every
infinite dimensional Banach space admits an infinite $(1+$)-separated set.
  Terenzi constructed an infinite dimensional Banach space
with no infinite equilateral subset (\cite{terenzi}). However, every infinite dimensional Banach space admits arbitrarily large finite
equilateral set by  results of Brass and Dekster \cite{brass, dekster}.
So the Ramsey property gives the positive answer to
infinite (in place of uncountable) version of Question $(S)$ and  the negative answer to
infinite (in place of uncountable) version of Question $(E')$ from the Introduction, while 
Terenzi's result shows that an infinite version of $(E)$ has the negative answer.

Concerning the versions of Question $(S')$ it  is clear that the unit spheres of separable Banach spaces 
can be covered by countably many sets of
arbitrarily small diameters so in this sense they are trivially dichotomous. But if we ask 
about dichotomies with finite collections of small diameter sets we have 
a theorem of
Furi and Vignoli (\cite{furi}) based on applying Lusternik-Schnierelman theorem and Dvoretzky theorem which
says 
that   $S_\X$ is not the union of finitely many sets of diameters less than $2$ for any infinite dimensional Banach space $\X$.

Now we pass to the Ramsey properties of the uncountable above the continuum.  We have the following classical: 

\begin{theorem}[Erd\"os-Rado]\label{erdos-rado} 
Suppose that $\kappa, \lambda$ are infinite cardinals satisfying $2^\kappa<\lambda$.
Let $X$ be a set of cardinality $\lambda$ and $c:[X]^2\rightarrow\kappa$.
There is an $Y\subseteq X$ of cardinality $\kappa^+$ such that $c[[Y]^2]$ is a singleton.
\end{theorem}

By an argument similar to the proof of Proposition \ref{omega-ramsey}  using the 
Erd\"os-Rado Theorem \ref{erdos-rado} one obtains the following:

\begin{proposition}\label{erdos-rado-app} Suppose  that $M$ is infinite, $d: [M]^2\rightarrow \R_+$.
\begin{enumerate}
\item For every $\varepsilon>0$ every subset of $M$ of cardinality bigger than
$2^\omega$ contains an uncountable $\varepsilon$-approximately equilateral subset.
\item For every $r>0$ every subset of $M$ of cardinality bigger than
$2^\omega$ contains an uncountable $(r+)$-separated subset or
an uncountable set of diameter not bigger than $r$.
\item Every subset of $M$ of cardinality bigger than
$2^{2^\omega}$ contains an equilateral subset of cardinality bigger than $2^\omega$.
\end{enumerate}
\end{proposition}
Item (3) above is noted in \cite{terenzi}.  A  deep analysis of the impact of
the Erd\"os-Rado Theorem \ref{erdos-rado} on the existence of uncountable $(1+)$-separated sets
in the unit spheres of Banach spaces of densities bigger than $2^\omega$ is presented in \cite{hkr}.

The only known so far examples of Banach spaces
whose unit spheres do not admit uncountable ($1+$)-separated sets or uncountable equilateral sets
have densities up to $2^\omega$
(\cite {pk-kottman, pk-hugh}).
It is not known at the moment if there are Banach spaces of density in the interval
$(2^\omega, 2^{2^\omega}]$ with no uncountable (infinite) equilateral sets.

However, in this paper we focus on nonseparable Banach spaces of densities not bigger than $2^\omega$
and so we are interested in Ramsey-type properties of such uncountable cardinals. 
In general,  natural versions of the Ramsey theorem already fail for the first uncountable cardinal
(see Theorem \ref{sierpinski1}) which has many consequences in our geometric context as indicated
in the following subsection. However  more delicate Ramsey-type results may hold for small
uncountable cardinals especially when we are accepting a  mere consistency as opposed to
provable theorems.
This has been investigated in depth by Todorcevic and many others.  One of
the fundamental results we will use in Section 3 is the following:

\begin{theorem}[Todorcevic]\label{todorcevic} 
The Open Coloring Axiom {\rm(}{\sf OCA}{\rm)} is consistent, where {\sf OCA} stands for
the following sentence: If $M$ is a second countable regular space
and $c:[M]^2\rightarrow\{0,1\}$ is such that $\{(x, y)\in M\times M: c(\{x, y\})=0\}$
is open in $M\times M$, then either 
$M=\bigcup_{n\in \N}M_n$ with $c[[M_n]^2]=\{1\}$ for each $n\in \N$ or 
else there is an uncountable $N\subseteq M$ such that $c[[N]^2]=\{0\}$.
\end{theorem}

For more information on the above three theorems see e.g. \cite{jech} or \cite{stevo-problems}.
Often OCA is stated for subsets of the reals. It is well known
that it is equivalent to its version for separable metrizable spaces and so
for second countable regular spaces by metrization theorems (see e.g. \cite{farah}).

It turns out that the above Ramsey-type property  holds in ZFC if we  assume 
a relatively low descriptive complexity of the domain of the coloring, namely that it is analytic.
Recall that a set is analytic if it is a continuous image of a Borel subset of 
a Polish space. For us the main context will be Borel and analytic subsets
of $\ell_\infty$ considered with the Polish topology inherited from $\R^\N$. For
more on these notions see \cite{kechris}.
We note that the unit sphere $S_{\ell_\infty}$ is a $G_\delta$ set with this topology
as this is the intersection of all the sets $V_{n, m}=\{x\in \R^\N: -1-1/n<x(m)<1+1/n\}$
for $n, m\in \N$ with $U=\bigcap_{n\in \N}\bigcup_{m\in \N}\{x\in \R^\N: |x(m)|>1-1/n\}$.

\begin{theorem}[Feng \cite{feng}]\label{feng} 
 If $M$ is an analytic set 
and $c:[M]^2\rightarrow\{0,1\}$ is such that $\{(x, y)\in M\times M: c(\{x, y\})=0\}$
is open in $M\times M$, then either 
$M=\bigcup_{n\in \N}M_n$ with $c[[M_n]^2]=\{1\}$ for all $n\in \N$ or 
else there is a perfect $N\subseteq M$ such that $c[[N]^2]=\{0\}$.
\end{theorem}

Another Ramsey-type property of the uncountable less than $2^\omega$ 
holds under Martin's axiom. Recall that a finite subset  $F$ of
a partial order  $\PP$ is   compatible if  there is $q\in \PP$ such that $q\leq p$ for
every $p\in F$; A subset of $\PP$ is called centered if its every finite subset is compatible.

\begin{theorem}\label{ma} Assume {\sf MA}. Suppose that $\PP$ is a partial order of cardinality
less than $2^\omega$. Either $\PP$ admits an uncountable subset 
of pairwise incompatible elements or $\PP=\bigcup_{n\in \N}\PP_n$, where
each $\PP_n$ is centered.
\end{theorem}

The above is  a  type of dichotomy as in Definition \ref{def-dichotomous}, and  it will be used this
way in Section 5 (cf. the last section of \cite{equi}). In fact the above dichotomy
for partial orders is equivalent to {\sf MA} as proved in \cite{stevo-boban} (cf. section 7 of \cite{stevo-problems}).
We also note in passing that  it is possible to express {\sf MA} itself more directly in the language of 
partitions of finite sets (see \cite{stevo-boban}, \cite{stevo-problems}).

\subsection{The failure of the Ramsey property and some of its consequences.}

In this subsection we recall how the Ramsey property fails for small uncountable cardinals
and see some direct consequences of this in our geometric context (see the previous subsection
for more explanations on the Ramsey-type properties of the uncountable).

\begin{theorem}[Sierpi\'nski]\label{sierpinski1} There is $c: [2^\omega]^2\rightarrow \{0,1\}$
 without uncountable monochromatic set.
\end{theorem}
\begin{proof} This is Sierpi\'nski coloring where $c(\{r, s\})=1$ if and only if the usual order agrees on $\{s, t\}$
with $\prec$, where $\prec$ is a well-ordering of the reals.
\end{proof}

\begin{theorem}[Sierpi\'nski]\label{sierpinski2} There is $c: [2^{2^\omega}]^2\rightarrow [0,1]$ 
without  monochromatic three element set.
\end{theorem}
\begin{proof} Let $[0,1]=\{r_\xi: \xi<2^\omega\}$. Look at $2^{2^\omega}$
as at
$\{0,1\}^{2^\omega}$ and define $c(x, y)=r_\xi$ if $\xi=\min\{\alpha<2^\omega: x(\alpha)\not=y(\alpha)\}$.
\end{proof}

\begin{theorem}[Todorcevic \cite{stevo-acta}]\label{stevo} There is $c: [\omega_1]^2\rightarrow \omega_1$ 
such that $c[[A]^2]=\omega_1$ for every uncountable $A\subseteq\omega_1$.
\end{theorem}

\begin{theorem}[Hewitt, Marczewski, Pondiczery]\label{hmp}
Suppose that $X$ is a separable topological space containing at least two points and $\kappa$ is an infinite cardinal.
$X^\kappa$ is separable if and only if $\kappa\leq 2^\omega$.
\end{theorem} 

\begin{lemma}\label{hmp-lemma} 
For every uncountable cardinal $\kappa$ there is $c:[\kappa]^2\rightarrow \{0,1\}$
such that there is no uncountable $1$-monochromatic set and $\kappa$ is not
the union of countably many $0$-monochromatic sets.
\end{lemma}
\begin{proof} 
For $\kappa\leq 2^\omega$ use Theorem \ref{sierpinski1}. Now suppose that $\kappa>2^\omega$.
Instead of $\kappa$ consider the family  of finite partial functions from
$\kappa$ into $\{0,1\}$. For two such functions $p, q$ define $c(\{p,q\})=1$
 if and only if $p\cup q$ is not a function.
Since $\{0,1\}^\kappa$ satisfies the c.c.c., there is no uncountable $1$-monochromatic set.
On the other hand if all the family were the countable union of $0$-monochromatic sets, their unions
would define a dense countable subset of $\{0,1\}^\kappa$ which contradicts the hypothesis on $\kappa$
by the Marczewski, Hewitt, Pondiczery theorem \ref{hmp}.
\end{proof}

\begin{proposition}\label{antisubbanach} Suppose that $\kappa$ is an 
infinite cardinal and $A\subseteq [1,2]$
is of cardinality $\omega_1$.
There are Banach spaces whose unit spheres contain
\begin{enumerate}
\item $1$-separated set $\Y$ of cardinality $2^\omega$, 
 where distances between points are $1$ or $2$ such that for every  uncountable subset $\Z\subseteq \Y$
there are  two points in $\Z$ distant by $1$ and two points in $\Z$  distant by $2$.
So
${\mathsf K}^+(\Y)=[0,1)$,  ${\Sigma}(\Y)=\{2\}$ and $\Y$ is hyperlateral.
\item $1$-separated set $\Y$ of cardinality $\omega_1$, such that for every  
uncountable subset $\Z\subseteq \Y$ for every $a\in A$
there are  two points in $\Z$ distant by $a$.
\item  $1$-separated set $\Y$ of cardinality $2^{2^\omega}$ with no equilateral triangle.
\item $1$-separated set $\Y$ of cardinality $\kappa$ with no uncountable
$(1+)$-separated subset such that $\Y$ is not the countable union of sets of diameters less than $2$.
\end{enumerate}
\end{proposition}
\begin{proof}
The items  of the proposition follow from Theorems \ref{sierpinski1}, \ref{stevo}, \ref{sierpinski2} and 
Lemma \ref{hmp-lemma} respectively
when we apply Lemmas \ref{color-dist} and \ref{kuratowski}.
\end{proof}

\section{The spheres of Banach spaces with weakly$^*$ separable dual balls.}

In the first part of this section we show that
under the hypothesis of  {\sf OCA} or in {\sf ZFC} under a descriptive complexity
hypothesis about the Banach space $X\subseteq\ell_\infty$ the sphere $S_\X$ is dichotomous.
So, for example, for such spaces the only reason why the sphere may not admit an uncountable
$(1+)$-separated set is that it is the union of countably many sets of diameters not bigger than $1$
(Proposition \ref{oca>dich} and \ref{feng>dich}).
In fact the {\sf OCA} result implies that any uncountable subset of $\ell_\infty$ is dichotomous.
This could be compared with the fact that the unit sphere of $c_0(\omega_1)$
has a  non-dichotomous subset while it is dichotomous (Proposition \ref{c0-notherdich}).
In the second part of this section we show that the hypotheses of the above results cannot be weakened
and the conclusions cannot be strengthened in general: Consistently when {\sf OCA} fails,
the spaces of \cite{ad-kottman} which
are isometric to subspaces of $\ell_\infty$ are hyperlateral (Proposition \ref{ad-hyper})
and in {\sf ZFC}  the spaces providing the negative answers to questions (S) and (E) from
the first section  are Borel in $\R^\N$ (Proposition \ref{examples-analytic})
so being dichotomous cannot be strengthened in Proposition \ref{feng>dich} to
the existence of uncountable $(1+)$-separated sets or infinite equilateral sets.

\begin{lemma}[\cite{dancer}]\label{ball} Suppose  that $\X$ is a Banach space. 
The dual unit ball $B_{\X^*}$ of $\X^*$
is separable in the weak$^*$ topology if and only if $\X$ is isometric to a subspace of $\ell_\infty$.
\end{lemma}
\begin{proof} Let $D=\{x_n^*: n\in \N\}$ be a weakly$^*$ dense subset of the dual ball of $\X^*$. 
Define $F: \X\rightarrow \R^\N$ by $F(x)=(x^*_n(x))_{n\in \N}$.  By the Hahn-Banach theorem
it is an isometric linear embedding of $\X$ into $\ell_\infty$.
If $T:\X\rightarrow \ell_\infty$ is an isometry, use $T^*$.
\end{proof}

\begin{definition} Let $X$ be a topological space. A function $f: X\rightarrow \R$
is called lower semicontinuous (lsc) if $\{x\in X: f(x)>a\}$ is open for every $a\in \R$.
\end{definition}

\begin{definition}  Suppose that $\X$ is a Banach space. We say that its 
norm is second countable lsc on the unit sphere
if there is a second countable, regular topology $\tau$ on $S_\X$ such that 
$d: S_\X\times S_\X\rightarrow \R$ defined by $d(x, y)=\|x-y\|$ is lsc with respect to the 
product topology $\tau\times\tau$.
\end{definition}

\begin{lemma}\label{separable-lsc} Suppose that a Banach space $\X$ has separable dual ball in the weak$^*$ topology.
Then  its norm is second countable lsc on the unit sphere $S_\X$.
\end{lemma}
\begin{proof} 
By lemma \ref{ball} the space $\X$ can be identified with a subspace of $\ell_\infty$.
Consider $\ell_\infty$ as a subspace of $\R^\N$ with the product topology. It is second countable and regular.
If $x, x'\in S_\X$ and $\|x-x'\|_\infty>a$, then there is $n\in \N$ such that
$r=\|x(n)-x'(n)\|_\infty>a$ and so $\|y(n)-y'(n)\|_\infty>a$ for any 
$$(y, y')\in \{t\in \X: |t(n)-x(n)|<(r-a)/3\}\times \{t\in \X: |t(n)-x'(n)|<(r-a)/3\}.$$
So $\{(x, x')\in \X\times\X: \|x-x'\|>a\}$ is open in $\X\times\X$ in the
subspace topology inherited from $\R^\N\times\R^\N$ as required.
\end{proof}

We do not know if there are  Banach spaces $\X$ which have second countable lsc norm on its sphere but
the dual ball $B_{\X^*}$ is not weakly$^*$ separable. This may be related to adding a Polish topology
which makes a given coloring continuous by forcing as studied in \cite{geschke-potent} by Geshke.

\begin{proposition}\label{oca>dich} Assume {\sf OCA}. Suppose that the norm of a  Banach space $\X$ 
is second countable lsc on the unit sphere $S_\X$. Then any subset of $S_\X$ is dichotomous.
In particular all subsets of the unit spheres of Banach spaces with weakly$^*$ separable dual balls
are dichotomous.
\end{proposition}
\begin{proof} 
Consider a subset $\Y$ of $S_\X$  with the second countable regular topology which
exists by the hypothesis.  Fix $r\in [0,2]$. Since the distance is lsc with respect to this topology
we are in the position of applying the {\sf OCA} (Theorem \ref{todorcevic})
for the coloring $c:S_\X\times S_\X\rightarrow \{0,1\}$ given
by $c(\{x, x'\})=0$ if and only if $\|x-x'\|>r$.

So the {\sf OCA} implies that $\Y$ either admits an uncountable
$0$-monochromatic set, which yields an uncountable $(r+)$-separated subset
of $\Y$  or
$\Y$ is the union of countably many $1$-monochromatic sets, which yields a partition of 
$\Y$ into countably many sets of diameter not bigger than $r$.

The last part of the proposition follows from Lemma \ref{separable-lsc}.
\end{proof}

\begin{proposition}\label{oca-auerbach} Assume {\sf OCA}. Let $\X$ be a Banach 
subspace of $\ell_\infty$
\begin{enumerate} 
\item If $\X$ admits an uncountable Auerbach system, then $S_\X$ admits an uncountable
$(1+)$-separated set.  
\item For every $\varepsilon>0$ every uncountable separated subset of $\X$ admits its uncountable subset
which is $\varepsilon$-approximately equilateral.
\end{enumerate}
 \end{proposition}
\begin{proof}
Use Proposition
\ref{oca>dich} and Propositions \ref{dich>auerbach} and \ref{dich>ramsey}.
\end{proof}

\begin{proposition}\label{feng>dich} Suppose that $\X\subseteq \ell_\infty$ is a Banach space. Suppose that
$\Y\subseteq \X$ is analytic 
with respect to the topology
on $\R^\N$.  Then $\Y$ is dichotomous. In particular if $\X$ is analytic, then $S_{\X}$
is dichotomous.
\end{proposition}
\begin{proof}
As in the proof of Lemma \ref{separable-lsc} we note that for every $r>0$ the set
$\{(y, y')\in \Y\times \Y:  \|y-y'\|>r\}$ is open in $\Y\times \Y$ with the subspace 
topology inherited from the product topology on $\R^\N\times \R^\N$
and so Feng's Theorem \ref{feng} can be applied as in the proof of Proposition \ref{oca>dich}.

To obtain that $S_\X$ is analytic if $\X$ is analytic we note that 
$S_\X$ is a Borel subset of $\X$ as noted before Theorem \ref{feng}.

\end{proof} 

\begin{proposition}\label{feng-auerbach} Suppose that $\X\subseteq\ell_\infty$ is a Banach
space.
\begin{enumerate}
\item If $\X$ admits an uncountable Auerbach system $(x_\alpha, x^*_\alpha)_{\alpha<\omega_1}$
such that $\{x_\alpha: \alpha<\omega_1\}\subseteq \Y\subseteq S_\X$, where $\Y$
is an analytic subspace of $\R^\N$, 
then $S_\X$ admits an uncountable $(1+)$-separated set.
\item For every $\varepsilon>0$ every uncountable separated $\Y\subseteq \X$ which is  analytic in $\R^\N$
admits an uncountable $\varepsilon$-approximately equilateral subset.
\end{enumerate}
\end{proposition}
\begin{proof}
Use Proposition \ref{feng>dich} and Propositions   \ref{dich>auerbach} and \ref{dich>ramsey}.
\end{proof}

In the following part of this section we give examples of Banach spaces
which show that in some sense the hypotheses of Propositions \ref{oca>dich} and \ref{feng>dich}
cannot be weakened and their conclusions cannot be made stronger.
These are examples of \cite{pk-kottman}, \cite{ad-kottman} and \cite{pk-hugh}.
To show that they do the job, we need a couple of lemmas.

\begin{lemma}\label{span-Borel} Suppose that  $\Y\subseteq \ell_\infty$ is $\sigma$-compact as a subspace of $\R^\N$.
Then the norm closure of the span of $\Y$ in $\ell_\infty$ is the
countable intersection of $\sigma$-compact subsets of $\R^\N$ and hence it is Borel.
\end{lemma}
\begin{proof} All topological concepts without mentioning the norm refer to the 
product topology of $\R^\N$.
We will use the fact that the coordinatewise algebraic operations are continuous in $\R^\N$
and that $\ell_\infty$ is $\sigma$-compact as $\ell_\infty=\bigcup_{m\in \N}[-m, m]^\N$.

For $n\in \N$ let $F_n\subseteq \Y$ be compact  
such that $F_n\subseteq F_{n+1}$ and $\bigcup_{n\in \N}F_n=\Y$. 
Functions $s_n: \R^n\times (\R^\N)^n\rightarrow \R^\N$ defined
as 
$$s_n(a_1,\dots, a_n, x_1, \dots, x_n)=\sum_{1\leq k\leq n} a_kx_k$$ are continuous 
and 
$$span(\Y)=\bigcup_{n\in \N} s_n[[-n, n]^n\times (F_n)^n].$$
So $span(\Y)$ is $\sigma$-compact as well.  Hence,  we may already assume that
$\Y$ is a linear subspace of $\ell_\infty$. 
 The norm closure of $\Y$ is
$\bigcap_{n\in \N}E_n$, where
$$E_n=\{x\in \ell_\infty: \exists y\in \Y\  \    \|x-y\|_\infty\leq 1/n\}.$$
Defining 
$$D_n=\{(x, y)\in \ell_\infty\times\ell_\infty: \|x-y\|_\infty\leq 1/n\}$$ 
we obtain that 
$$E_n=\pi[D_n\cap (\ell_\infty\times \Y)],$$
where $\pi:\R^\N\times \R^\N\rightarrow\R^\N$ is the projection on the first coordinate.
Note that as $D_n$ is a relatively closed set of a $\sigma$-compact
$\ell_\infty\times\ell_\infty\subseteq \R^\N\times \R^\N$
and $\pi$ is continuous, it follows that $E_n$ is $\sigma$-compact as well and so the lemma follows.
\end{proof}

The following Lemma concerns certain equivalent renormings of 
subspaces of $\ell_\infty$.  More information on these renormings can be found
in the papers \cite{pk-kottman} and \cite{pk-hugh}.

\begin{lemma}\label{XT} Suppose that $\X$ is a Banach subspace 
 of $\ell_\infty$. 
Let $T: \ell_\infty\rightarrow \ell_2$ be  given by
$$T(x)=\Big({x(n)\over{2^{n\over 2}}}\Big)_{n\in \N}$$
for $x\in \ell_\infty$ 
and let $\X_T$ be the space $\X$ with
the equivalent norm $\|\  \|_{\infty, 2}$ given by
$$\|x\|_{\infty, 2}=\|x\|_\infty +\|T(x)\|_2.$$
Then $\X_T$ is isometric to a subspace $\Y_T$ of $\ell_\infty$.
If $\X$ is Borel as a subspace of $\R^\N$, then $\Y_T$
 is Borel as a subspace of $\R^\N$ as well.
\end{lemma}
\begin{proof}
For $n\in \N$ let $\tilde{n}$s be distinct elements so that $\tilde{\N}=\{\tilde{n}: n\in \N\}$ does
not intersect with $\N$.
Let $(A_n)_{n\in \N\cup\tilde{\N}}$ be a partition of $\N$ into infinite sets and
let $F: \R^\N\rightarrow \R^\N$ satisfy $F(x)(k)=x(n)$ if $k\in A_n$ and $F(x)(k)=-x(n)$
 if $k\in A_{\tilde{n}}$ for $n\in \N$.
It is clear that $F$ is a homeomorphism onto its closed image and moreover it is
a linear isometric embedding   when restricted to $\ell_\infty$.

Consider $G:\R^\N\rightarrow \R^\N$ which is linear isometric embedding when restricted to $\ell_\infty$ which sends
$x\in \R^\N$ to $G(x)$ such that for each $n\in \N$   the restriction
$G(x)|A_n$ is a copy of $x$ through a fixed bijection between $A_n$ and $\N$ and 
for each $n\in \N$  and  the restriction
$G(x)|A_{\tilde{n}}$ is a copy of $x$ through a fixed bijection between $A_{\tilde{n}}$ and $\N$. 
 It is clear that $G$ is a homeomorphism
from $\R^\N$ onto its image which is closed in $\R^\N$. 
Hence $(F, G):(\R^\N)^2\rightarrow (\R^\N)^2$ is a homeomorphism onto its closed image.
Let $S:\ell_2\rightarrow \ell_\infty$ be a linear isometric embedding
which exists by the Banach-Mazur theorem (as we note that $C([0,1])$ 
isometrically embeds into $\ell_\infty$). We define
$$\Y_T=\{F(x)+G(S\circ T(x)): x\in \X\}.$$
Now we prove that $(\Y_T, \|\ \|_\infty)$ is an isometric copy of $\X_T$, i.e., of 
 $(\X, \|\ \|_{\infty, 2})$.
 We consider
the embedding $\iota$ which sends $x\in \X$ to $F(x)+G(S\circ T(x))$. The linearity 
of $\iota$ follows since  all involved
functions are linear.   We need to see
that 
$$\|F(x)+G(S\circ T(x))\|_\infty=\|x\|_{\infty, 2}=\|x\|_\infty+\|T(x)\|_2.$$
 Clearly
$\|F(x)\|_\infty=\|x\|_\infty$ and $\|G(S\circ T(x))\|_\infty=\|T(x)\|_2$.
So we have $\|F(x)+G(S\circ T(x))\|_\infty\leq \|x\|_\infty+\|T(x)\|_2$ by the triangle inequality.
The other inequality follows from the fact
that if $|x(n)|$ for $n\in \N$ is close to $\|x\|_\infty$ and 
$|S\circ T(x)(m)|$ for $m\in \N$ is close to $\|T(x)\|_2$, we will find 
$k\in A_n$ or $k\in A_{\tilde{n}}$ such that $|F(x)(k)|=|x(n)|$ and $G(S\circ T(x))(k)=S\circ T(x)(m)$
and they have the same signs, so they approximate $\|x\|_\infty+\|T(x)\|_2$. 

It remains to prove the second part of the lemma, i.e., that $\iota[\X]\subseteq\R^\N$
is Borel if $\X\subseteq\R^\N$ is Borel. It is enough to prove
that $\iota$ 
restricted to $\X\cap [-m, m]^\N$ is continuous for each $m\in \N$.
Then  one can use the Lusin-Suslin Theorem (15.2 in \cite{kechris}) which says that
 images of  Borel sets under continuous injective functions between Polish spaces are Borel,
 and clearly $\iota$ is injective on the considered sets as we have proved that it is an isometry
 when restricted to $\X$.
 
 It is enough to see that $S\circ T$ is continuous with respect to
 the product topology in $\R^\N$ when restricted to $[-m, m]^\N$ as $F, G$ and $+$ are continuous.
 If $x_n, x\in [-m, m]^\N$ and $x_n$ converges to $x$ in the product
 topology, then  for sufficiently large $n$ we have first finitely many  
 coordinates of $x_n-x$ arbitrarily close to zero. Since other coordinates are
 bounded by $2m$, using the formula for $T$ we can get $\|T(x_n)-T(x)\|_2$ as close to zero  as we wish
 by taking $n$ sufficiently large. Then $\|S\circ T(x_n)-S\circ T(x)\|_\infty$
 is as close to zero as we wish, since $S$ is a linear isometry. This implies that $S\circ T(x_n)$s
 fall into a prescribed neighborhood of $S\circ T(x)$ with respect to the product topology.
 This completes the proof of the continuity of $\iota$ when restricted
 to the sets $[-m, m]^\N$ and the proof of the lemma.
\end{proof}

The following example shows that we cannot drop
the hypothesis of {\sf OCA} in Proposition \ref{oca>dich} or
the hypothesis of the analyticity in Proposition \ref{feng>dich}.

\begin{proposition}[\cite{ad-kottman}]\label{ad-hyper} Assume {\sf CH}. 
For every $\rho\in(0, 1)$ there is a nonseparable Banach space 
$\X_\rho$  whose dual ball is separable in the weak$^*$ topology  
such that for every $\delta>0$ there is  $\varepsilon>0$ such that every uncountable 
$(1-\varepsilon)$-separated subset of $S_\X$ contains two points distant by less than $1$ and
contains two points distant by more than $2-\rho-\delta$. In particular, $S_{\X_\rho}$ is hyperlateral
and ,
${\mathsf K}(S_{\X_\rho})=[0,1)$,  $\Sigma(S_{\X_\rho})\subseteq [2-\rho, 2]$.
\end{proposition}
\begin{proof}
This is the space of Theorem 3 of \cite{ad-kottman}.  
It is of the form $\X_T$ as described in Lemma \ref{XT},
so $\X_T$ is isometric to a subspace of $\ell_\infty$ by 
the first part of Lemma \ref{XT} and so its dual ball is weakly$^*$ separable by Lemma \ref{ball}.
The values 
${\mathsf K}(S_{\X_\rho})=[0,1)$ and  $\Sigma(S_{\X_\rho})\subseteq [2-\rho, 2]$
follow from the stated property of $(1-\varepsilon)$-separated sets
in the unit sphere which is the statement of Theorem 3 of \cite{ad-kottman}.
The sphere is hyperlateral by Proposition \ref{sphere-no} (5) and the above
properties of the sphere.
\end{proof}

Another consistent example of Banach space isometric to a subspace of $\ell_\infty$
 whose unit sphere is hyperlateral  (but admits an
 uncountable $(1+)$-separated sets) and  is far from dichotomous
is presented in Proposition \ref{hyper-ck}.

The following examples (from \cite{pk-kottman} and \cite{pk-hugh}) show that in Proposition
\ref{feng>dich}
we cannot strengthen the conclusion asserting that
the unit spheres of the spaces admit uncountable $(1+)$-separated sets 
or uncountable Auerbach systems or even infinite equilateral sets  even 
in the case of Borel subsets of $\R^\N$.

\begin{proposition}[\cite{pk-kottman, pk-hugh}]\label{examples-analytic}
There are Banach subspaces $\X$ and $\Y$ of $\ell_\infty$ which are
Borel subsets of $\R^\N$ such that
\begin{enumerate}
\item $S_\X$ does not admit an uncountable Auerbach system nor an uncountable
$(1+)$-separated set nor  an uncountable equilateral set.
\item $\Y$ does not admit an infinite equilateral set.
\end{enumerate}
\end{proposition}
\begin{proof}
Both of the  spaces of \cite{pk-kottman} and \cite{pk-hugh} are
of the form $\X_T$ as described in Lemma \ref{XT}. The properties stated in
the above proposition are the main results of \cite{pk-kottman} and \cite{pk-hugh}.
So it remains to prove that the isometric spaces of the form $\Y_T$
as defined  in Lemma \ref{XT} are Borel subsets of $\R^\N$.
For this we use the second part of Lemma \ref{XT}, so it remains to prove
that the original subspaces of $\ell_\infty$ (i.e., 
before the renorming $\|\ \|_{\infty,2}$) are Borel.
This in turn, will follow from Lemma \ref{span-Borel} if we can prove
that some generating subsets of the original spaces are compact subsets of
 $\{0,1\}^\N\subseteq\R^\N$.

In the case of item (1), that is the space of \cite{pk-kottman}, 
the original space is the Johnson-Lindenstrauss space $\X_\A$
generated by $c_0\cup\{1_A: A\in \A\}$, where
$\A$ is an uncountable almost disjoint family which is $\R$-embeddable.
Characteristic functions of such a  family indeed can form a  closed subset of
$\{0,1\}^\N$. This is, for example, proved in Lemma 30 of \cite{supermrowka}.
It should be clear that the almost disjoint family considered there (i.e,
the family of all branches of the Cantor tree)  is $\R$-embeddable. This is because
fixing a bijection $\phi: \N\rightarrow \bigcup_{n\in \N}\{0,1\}^n$
we can consider embedding  of $\N$ into $\Q$ which sends
an integer $k=\phi^{-1}(s)$ for $s\in \{0,1\}^n$ to 
$\sum_{i<3n}{y_s(i)\over3^{i+1}}$, where $y_s(i)$ is equal $s(i/3)$
if $i$ is divisible by $3$ and $y_s(i)=j$ if the remainder
in the division of $i$ by $3$ is $j+1$ for $j\in \{0,1\}$.
Also $c_0$ is generated by a closed convergent sequence with its limit in $\{0,1\}^\N$
of characteristic functions of singletons together with the zero function.

In the case of item (2) the original space is $\ell_1([0,1])$.
So it is enough to prove that it is isometric to a 
subspace of $\ell_\infty$ which is generated by a closed  subset of $\{-1,1\}^\N\subseteq\R^\N$.
Recall that a family $\F$ of subsets of $\N$ is called
independent if 
$$\bigcap\{A: A\in \F'\}\cap\bigcap\{\N\setminus A: A\in \F''\}\not=\emptyset$$
for any finite disjoint $\F', \F''\subseteq \F$.
It is also well known that 
$\{1_A-1_{\N\setminus A}: A\in \F\}$ generate in $\ell_\infty$
a subspace isometric to $\ell_1(\F)$ if $\F$ is an independent family.
Namely we define an isometric embedding $\iota:\ell_1(\F)\rightarrow\ell_\infty$ by
the linearity and sending 
 $\ell_1(\F)\ni1_{\{A\}}: \F\rightarrow\{0,1\}$ to
 $1_A-1_{\N\setminus A}$. Then 
 $$\|\iota(\sum_{i<n} a_i1_{\{A_i\}})\|_\infty=
 \|\sum_{i<n} a_i(1_{A_i}-1_{\N\setminus {A_i}})\|_\infty\leq \sum_{i<n}|a_i|=
 \|\sum_{i<n} a_i1_{\{A_i\}}\|_1$$
 by the triangle inequality. Picking 
 $$k\in \bigcap\{A_i: a_i\leq0\}\cap\bigcap\{\N\setminus A_i: a_i>0\}$$
 we conclude that 
 $$\|\iota(\sum_{i<n} a_i1_{\{A_i\}})\|_\infty\geq
 \|\sum_{i<n} a_i(1_{A_i}-1_{\N\setminus {A_i}})(k)\|_\infty=\sum_{i<n}|a_i|=\|\sum_{i<n} a_i1_{\{A_i\}}\|_1.$$
So we conclude that $\iota$ is an isometry as required.
So it remains to know that there is an independent family of subsets of $\N$
of cardinality $2^\omega$ whose characteristic functions form a closed subset of $\{0,1\}^\N\subseteq\R^\N$. 
The standard
independent family due to Fichtenholtz and Kantorovich of \cite{fk}  actually does the job.
It is a family of  the form $\{A_x: x\in 2^\N\}$ of subsets 
of the countable set $\mathcal O=\bigcup_{n\in \N}\{0,1\}^{2^n}$ (identified with $\N$
to get the family of subsets of $\N$), where the union is considered disjoint.
We define $A_x=\bigcup_{n\in \N}U_{x|n}$, where $U_{x|n}=\{s\in \{0,1\}^{2^n}: s(x|n)=1\}$.
One checks that if $v\in \R^{\mathcal O}$
is not any of
$1_{A_x}\in\{0,1\}^{\mathcal O}\subseteq\R^{\mathcal O}$, 
then it is verified by looking at finitely many coordinates in $\mathcal O$,
so the complement of $\{1_{A_x}: x\in 2^\N\}$ is open.

Again the application of Lemmas \ref{span-Borel} and\ref{XT}
proves that the space of the second  item is Borel in $\R^\N$.
\end{proof}

We note that a particular case of Banach spaces which embed
into $\ell_\infty$ as analytic subsets of $\R^\N$ are of the form
$C(K)$ for $K$ separable Rosenthal compacta (\cite{godefroy}).
Kania and Kochanek proved (Proposition 4.7 of \cite{tt}) that the unit spheres of nonseparable Banach spaces
from this class do admit uncountable $(1+)$-separated sets. 
So for this subclass the conclusion of  Proposition \ref{feng>dich} can be more specific
i.e., mention just the existence of an uncountable $(1+)$-separated subset of the unit sphere.

Recall that
Godefroy and Talagrand showed in \cite{godefroy-tal} that a nonseparable Banach space
isomorphic to a Banach subspace of $\ell_\infty$ which is analytic in $\R^\N$ admits
an uncountable biorthogonal system.  By Proposition \ref{examples-analytic}
 we see that the isometric counterpart (i.e., for Auerbach system in place of
 biorthogonal system) of Godefroy and Talagrand result does not hold.
 
 One could be tempted
to consider the impact of  the game determinacy on results like Proposition \ref{feng>dich} for more complex
subspaces of $\ell_\infty$ as, for example it was done by Godefroy and Louveau in \cite{determinacy}
in the case of biorthogonal systems. However  this is superseded by Proposition \ref{oca>dich} 
which concerns all subsets of $\ell_\infty$.

As noted above, in the proof of Proposition \ref{examples-analytic},
the spaces $\ell_1(\kappa)$ for $\kappa\leq 2^\omega$ can be isometrically embedded in $\ell_\infty$ and so
we should consider them in this section.
However, their unit spheres  trivially admit a $2$-separated sets of cardinality $\kappa$, so we obtain:

\begin{proposition}\label{l1} Suppose that  $\kappa\leq 2^\omega$ is an uncountable cardinal.
Then ${\mathsf K}^+(S_{\ell_1(\kappa)})=[0, 2)$ and
${\Sigma}(S_{\ell_1(\kappa)})=\{2\}$, hence $S_{\ell_1(\kappa)}$ is
dichotomous.
\end{proposition}

\section{The spheres of some  WCG Banach spaces}

In this section we prove that unit spheres of Banach spaces from
some classes of weakly compactly generated spaces are almost dichotomous.

\subsection{The spheres of $\ell_p(\lambda)$ for $1<p<\infty$}

Here we generalize result
of  Erd\"os and Preiss of (\cite{erdos-p}) which says  that
the unit sphere of nonseparable Hilbert space $\ell_2(\kappa)$ for $\omega_1\leq \kappa\leq 2^\omega$ 
is the union
of countably many sets of diameters not bigger than $\sqrt2+\varepsilon$ for any $\varepsilon>0$ and
the result of Preiss and R\"odl which says that it
is not the countable union of sets of diameters not bigger than $\sqrt2$ obtaining
analogous results for $\ell_p(\kappa)$ with
$\sqrt[p]2$ in place of $\sqrt2$.
We additionally observe that such spaces do not admit
uncountable $(\sqrt[p]{2}+)$-separated sets (but of course, admit an
uncountable $\sqrt[p]{2}$-equilateral set).
This is related to a classical result of Kottman (\cite{kottman}) which says
that $\ell_p$-spaces do not admit an infinite $(\sqrt[p]{2}+\varepsilon)$-separated
set for any $\varepsilon>0$. The conclusions are stated in Proposition
\ref{lp-dich}.

\begin{lemma}\label{pytha} Let $1<p<\infty$. 
Suppose that $x\in {\ell_p}$  and $supp(x)=a\cup b$ with $a\cap b=\emptyset$.
Then $\|x\|_p^p=\|x|a\|_p^p+\|x|b\|^p_p$.
\end{lemma}
\begin{proof}
$\|x\|_p^p=\Sigma_{i\in supp(x)}|x(i)|^p=\Sigma_{i\in a}|x(i)|^p+
\Sigma_{i\in b}|x(i)|^p=\|x|a\|_p^p+\|x|b\|^p_p$.
\end{proof}

\begin{lemma}\label{ellp-delta} Let $1<p<\infty$. 
Suppose that $x, y\in S_{\ell_p}$ satisfy $x|a=y|a$,  where $a=supp(x)\cap supp(y)$
with $\delta=\|x|a\|_p=\|y|a\|_p$.
Then $$\|x-y\|_p=\sqrt[p]{2(1-\delta^p)}\leq \sqrt[p]{2}.$$
\end{lemma}
\begin{proof}
Let $b=supp(x)\setminus a$ and $c=supp(y)\setminus a$. 
By Lemma \ref{pytha} we have $\|x|b\|^p_p=\|y|c\|^p_p=1-\delta^p$. Again by
Lemma \ref{pytha}  we get $\|x-y\|^p_p=2(1-\delta^p)$.
\end{proof}

\begin{proposition}\label{lp-union} Let $1< p<\infty$. For
every $\varepsilon>0$ the unit sphere
$S_{\ell_p(2^\omega)}$ is the union of countably many sets of diameters
not bigger than $\sqrt[p]{2}+\varepsilon$.
\end{proposition}
\begin{proof} Fix $\varepsilon>0$ and $1< p<\infty$. 
Let $D\subseteq\R$ be a countable set containing all rationals and closed under
addition, multiplication, division, taking powers with exponent $p$ and taking $p$-th roots.
It should be clear that  the set
$\Y$ of all elements of the unit sphere of $\ell_p(2^\omega)$ with finite supports and assuming
values in $D$ is dense in the unit sphere.
It is enough to show that
the set $\Y$ 
is the union of countably many sets of diameter not bigger than $\sqrt[p]{2}$.
Let $E=\{e_n: n\in \N\}\subseteq D^{(2^\omega)}$ be a dense subset of  $D^{(2^\omega)}$ 
where $D$ is considered with discrete topology,
which exists by the Hewitt-Marczewski-Pondiczery Theorem \ref{hmp}.
It follows that for each element $y\in \Y$ there is $n_y\in \N$ such that $e_{n_y}|supp(y)=y|supp(y)$.
We claim that $\Y_n=\{y\in \Y: n_y=n\}$ has diameter not bigger than $\sqrt[p]{2}$.
This follows from Lemma \ref{ellp-delta}.
\end{proof}

\begin{proposition}\label{lp-nonunion} Let $1<p<\infty$.  The sphere
$S_{\ell_p(\omega_1)}$ is not the union of countably many sets of
diameters not bigger than $\sqrt[p]{2}$.
\end{proposition}
\begin{proof} Fix $1<p<\infty$. Suppose that $\bigcup \X_n=S_{\ell_p(\omega_1)}$.
Let $X\subseteq \N$ be the set of all such $n\in\N$  that for every $\alpha<\omega_1$ there is $x\in \X_n$
with $\alpha\cap supp(x)=\emptyset$. Clearly $X\not=\emptyset$.  Using the fact that 
the supports of elements of $\ell_p(\omega_1)$ are countable for every $\alpha<\omega_1$ one can
find $x_n^\alpha\in \X_n$  satisfying $supp(x^\alpha_n)\cap\alpha=\emptyset$ for $n\in X$
 such that $supp(x_n^\alpha)\cap supp(x_m^\alpha)=\emptyset$
for every distinct $n, m\in X$.
For $\alpha<\omega_1$ 
construct $x_\alpha\in S_{\ell_p(\omega_1)}$ such that $supp(x_\alpha)=\bigcup_{n\in X}supp(x_n^\alpha)$
(in particular $supp(x_\alpha)\cap\alpha=\emptyset$)
and $x_\alpha|supp(x_n^\alpha)=\delta_n^\alpha x_n^\alpha|supp(x_n^\alpha)$ for some $\delta_n^\alpha<0$. 
Note that by Lemma \ref{pytha} we have
$$1=\|x_\alpha\|^p=|\delta_n^\alpha|^p+
\Big\|x_\alpha|\big(supp(x_\alpha)\setminus(supp( x_n^\alpha)\big)\Big\|^p_p.$$
It follows from Lemma \ref{pytha} that
 $\|x_n^\alpha-x_\alpha\|^p=(1+|\delta_n^\alpha|)^p+(1-|\delta_n^\alpha|^p)>2$
 as $(1+x)^p>1+x^p$ for $x>0$, since the derivative of the difference is
 positive for $x>0$ as $(1+x)^{p-1}>x^{p-1}$ for $x>0$ since $p>1$.
 As $x_\alpha$s have supports above $\alpha<\omega_1$ uncountably
many of them must belong to some $\X_n$ with $n\in X$. This shows that $\X_n$
has diameter strictly bigger than $\sqrt[p]{2}$.
\end{proof}

Preiss and R\"odl proved the above result for $\ell_2(\kappa)$ using another
 proof based on the Baire category theorem which also goes through for all $1<p<\infty$.

\begin{lemma}\label{lp-nosep} Let $1< p<\infty$. There is no uncountable 
$(\sqrt[p]{2}+)$-separated set in $S_{\ell_p(\omega_1)}$.
\end{lemma}
\begin{proof}
Let $\{x_\xi: \xi<\omega_1\}\subseteq S_{\ell_p(\omega_1)}$. We will show 
that there are $\xi<\eta<\omega_1$ such that
$\|x_\xi-x_\eta\|\leq\sqrt[p]{2}$. First note
that by passing to an uncountable set we may assume that there is $\alpha_0<\omega_1$ and $\delta>0$ such that
$\|x_\xi|\alpha_0\|^p>\delta$ for every $\xi<\omega_1$. Indeed, otherwise
we can choose an uncountable subset of $\{x_\xi: \xi<\omega_1\}$ whose elements have
pairwise disjoint supports. For such elements we have  $\|x_\xi-x_\eta\|^p=\|x_\xi\|^p+\|x_\eta\|^p=2$
as required by Lemma \ref{pytha}.
%Since $\ell_p(\alpha_0)$ is separable, we will additionally assume  that
%$\|x_\xi|\alpha_0-x_\eta|\alpha_0\|\leq \delta/3$.

 We claim that there is
$\alpha_0<\alpha<\omega_1$ such that for every $\varepsilon>0$ and every $\alpha<\beta<\omega_1$
there are uncountably many $\xi<\omega_1$ with $\|x_\xi|[\alpha, \beta)\|_p<\varepsilon$.
To prove the claim note that otherwise one can construct  by transfinite induction
$(\alpha_\gamma:\gamma<\omega_1)$ and $(\varepsilon_\gamma: \gamma<\omega_1)$  such that
$\|x_\xi|[\alpha_\gamma, \alpha_{\gamma+1})\|_p\geq\varepsilon_\gamma$ for
all but countably many $\xi<\omega_1$. As $\varepsilon_\gamma$ must be separated from
zero for uncountably many $\gamma<\omega_1$ this allows to find $\xi<\omega_1$
and pairwise disjoint intervals $I_i$ for $i<k$ such that $\|x_\xi|I_i\|\geq\varepsilon$
for some $\varepsilon$ and $k>\varepsilon^{-p}$. As $\|x_\xi\|^p\geq \Sigma_{i<k}\|x_\xi|I_i\|^p$
this would contradict the hypothesis that $\|x_\xi\|=1$. 

Using the claim and the fact that $\delta>0$
and by passing to an uncountable subset of $\{x_\xi: \xi<\omega_1\}$
we may assume that 
$$\|x_\eta|(\beta_\xi\setminus\alpha)\|\leq \varepsilon=\sqrt[p]{1+\delta/2}-1$$
for every $\xi<\eta<\omega_1$, where $\beta_\xi=\sup(supp(x_\xi))+1$. Now use the separability of
$\ell_p(\alpha)$ to pick $\xi<\eta<\omega_1$ such that 
$$\|x_\xi|\alpha-x_\eta|\alpha\|_p^p\leq \delta/2.$$
Finally let us estimate $\|x_\xi-x_\eta\|$ using its projections
onto $[0,\alpha)$, $[\alpha, \beta_\xi)$, $[\beta_\xi, \omega_1)$ and
the fact that $\|x\|_p^p=\|x|A\|_p^p+\|x|B\|_p^p+\|x|C\|_p^p$
for all $x\in \ell_p(\omega_1)$ and any partition of $\omega_1$ into sets $A, B, C$.
$$\|x_\xi-x_\eta\|^p_p\leq\delta/2 +(1+\varepsilon)^p+(1-\delta)= 2$$
as required.
\end{proof}

\begin{proposition}\label{lp-dich} Suppose that $1<p<\infty$ and $\kappa\leq 2^\omega$ is uncountable.
Then 
\begin{itemize}
\item ${\mathsf K}^+(S_{\ell_p(\kappa)})=[0, \sqrt[p]{2})$,
\item ${\Sigma}(S_{\ell_p(\kappa)})=(\sqrt[p]{2}, 2]$.
\end{itemize}
Hence $S_{\ell_p(\kappa)}$ is almost
dichotomous and not dichotomous. 
\end{proposition}
\begin{proof} Use Propositions \ref{lp-union}, \ref{lp-nonunion} and Lemma \ref{lp-nosep} and
the fact that the characteristic functions of singletons form
an uncountable $\sqrt[p]{2}$-equilateral set.
\end{proof}

\subsection{The spheres of $L_1(\mu)$ spaces}

\begin{lemma}\label{Lp-separated} Suppose that  $\kappa$ is an uncountable cardinal.
Then the unit sphere of $L_1(\mu)$, where $\mu$ is the product
measure on $\{0,1\}^\kappa$, admits
a $(2-\varepsilon)$-separated set of cardinality $\kappa$. Therefore 
${\mathsf K}^+(S_{L_1(\{0,1\}^\kappa)})=[0, 2)$,
 ${\Sigma}(S_{L_1(\{0,1\}^\kappa)})=\{2\}$ and the unit sphere
of $L_1(\mu)$ is dichotomous.
\end{lemma}
\begin{proof}
First note that if $A, B\subseteq \kappa$ are pairwise disjoint of cardinality $n\in \N$
and 
$f=2^{n}\chi_X$,
$g=2^{n}\chi_Y$, 
where $X=\{x\in \{0,1\}^\kappa: \forall\alpha\in A\ x(\alpha)=1\}$, 
and $Y=\{x\in \{0,1\}^\kappa: \forall\alpha\in B\ x(\alpha)=1\}$.
 Then 
$$\|f\|_{L_1}, \|g\|_{L_1}={\int |f|d\mu=(2^{n})(1/2^n)}=1.$$
On the other hand 
$$\|f-g\|_{L_1}=2\cdot{(2^{n})(1/2^n)(1-{1\over2^n})}={2-(1/2^{n-1})}.$$
So by choosing pairwise disjoint families of finite subsets
 of $\kappa$ we obtain $({2}-\varepsilon)$-separated
sets of cardinality $\kappa$. 
\end{proof}

\subsection{The spheres of $L_p(\mu)$ spaces for $1<p<\infty$}

\begin{proposition}\label{Lp} Suppose that  $\kappa\leq2^\omega$ is an uncountable cardinal
and $1<p<\infty$.
 The unit sphere of $L_p(\mu)$, where $\mu$ is the product
measure on $\{0,1\}^\kappa$, is almost dichotomous.
\end{proposition}
\begin{proof}
For a finite $A\subseteq\kappa$ let $\F_A$  denote the class of all real functions on $\{0,1\}^\kappa$
which depend on $A$, i.e., such $f$ that there is
$F:\{0,1\}^A\rightarrow \R$ such that
$f(x)=F(x|A)$ for every $x\in\{0, 1\}^\kappa$. 
For an $A\subseteq\kappa$ of the cardinality $n\in \N$ let $\sigma_A:n\rightarrow A$
be the order preserving bijection.  
We can define an isometry $T_A$ from  $L_p(\{0,1\}^n)$ into $\F_A$ considered
with the norm from $L_p(\{0,1\}^\kappa)$ by $T_A(F)(x)= F(x\circ\sigma_A)$
for any $F\in L_p(\{0,1\}^n)$ and $x\in \{0,1\}^\kappa$. 
Let $$c_n=\sup\{\|T_A(F)-T_B(F)\|_{L_p}: F\in L_p(\{0,1\}^n), \|F\|_{L_p}=1,$$
$$ A, B\in [\kappa]^n,
(\sigma_B\circ\sigma_A^{-1})|(A\cap B)=Id_{A\cap B}\}.$$
It is clear that to obtain $c_n$ one needs to consider finitely many pairs $A, B$ 
of  subsets of $\kappa$ of cardinality $n$ which exhaust the positions
of their intersections. Put $c=\sup_{n\in \N}c_n$.
By considering  uncountable families of subsets of $\kappa$ of cardinality $n\in \N$
where any two $A, B$ satisfy $(\sigma_B\circ\sigma_A^{-1})|(A\cap B)=Id_{A\cap B}$
for every $\varepsilon>0$ we can obtain uncountable
$(c_n-\varepsilon)$-separated subsets of $S_{L_p}$, so
Kottman's $(+)$-constant of $S_{L_p}$ is not smaller than $c$.
To complete the proof that $S_{L_p}$ is dichotomous it is enough to
show that for every $\varepsilon>0$ the unit sphere $S_{L_p}$ is the countable union of 
sets of diameter $c+\varepsilon$.  The remainder of the proof is devoted to this end.

Let $D$ be a countable subset of the reals which contains the rationals
and is closed under taking sums, products, powers of exponent $p$ and $p$-th roots.
By Luzin's theorem and Tietze extension theorem continuous functions on $\{0,1\}^\kappa$ are dense in $L_p(\mu)$.
 By the Stone-Weierstrass theorem continuous functions which depend on finitely many coordinates
and have rational values are dense with respect to the supremum norm in the set of all continuous functions
and so are dense in $L_p(\mu)$. Moreover the set $\Y$ of all functions which depend on finitely many coordinates (all
of them are continuous on $\{0,1\}^\kappa$) which are in the unit sphere of $L_p(\{0,1\}^\kappa)$
and have  values in $D$ are dense with respect to the supremum norm in the set of all continuous functions
and the unit sphere and so are dense in $S_{L_p(\mu)}$. It is enough to
prove that $\Y$ is the countable union of sets of diameters not bigger than $c$.

For $n\in \N$ let 
$$\F_n=\{f\in\F_A: A\in [\kappa]^{n}, \|f\|_{L_p}=1, f[\{0,1\}^\kappa]\subseteq D\}.$$
So it is enough to prove for each $n\in \N$
that $\F_n=\bigcup_{k,l\in \N} \F_{n,k, l}$ and the $L_p$-diameter of each $\F_{n,k, l}$ for $k, l\in \N$
is not bigger than $c_n$.

Fix $n\in \N$ and enumeration $(F_k)_{k\in \N}$ of all functions from $\{0,1\}^n$ into $D$.
Let $E=\{e_l: l\in \N\}$ be a countable dense set in $\N^\kappa$
where $\N$ is considered with discrete topology. It exists by Theorem \ref{hmp}. 

Given $f\in \F_A$ in the unit sphere with all  values in $D$ for a finite $A\in[\kappa]^n$ find $l\in \N$
such that the function $\sigma_A^{-1}$ enumerating elements of $A$ agrees with $e_l$.
Now find $k\in \N$ such that $T_A(F_k)=f$.
Put such $f$ into $\F_{n, k, l}$.  This completes the definition of $\F_{n, k, l}$.

Now suppose that $f, f'\in \F_{n, k, l}$ for some $n, k, l\in \N$. Let
$A, B\in [\kappa]^n$ be such that $f\in \F_A$ and $f'\in \F_B$.  By the construction
we have $T_A(F_k)=f$ and $T_B(F_k)=f'$ and $(\sigma_B\circ\sigma_A^{-1})|(A\cap B)=Id_{A\cap B}$
as the functions enumerating $A$ and $B$ in the increasing order agree on $A\cap B$ 
as they agree with $e_l$. So $\|f-f'\|_{L_p}$ is not bigger than $c_n$ by the definition
of $c_n$, as required.
\end{proof}

We do not know the exact value of $c_n$ or of $c$ in the above proof or if $S_{L_p}$
is dichotomous. Note that ${\mathsf k}^+(S_{L_p})$ is not $\sqrt[p]{2}$ in general
because we can construct in $L_p(\{0,1\}^{\omega_1})$ uncountable
sets of independent variables assuming values $1$ and $-1$ with equal probability $1/2$.
They have norm $1$ but their differences have the norm equal to $2^{1-1/p}$ which is strictly bigger
than $\sqrt[p]{2}$ if $p$ is sufficiently large.

\subsection{The spheres of Shelah-Wark spaces}

In \cite{shelah} Shelah considered a nonseparable version of a Schreier space whose
norm was induced by the anti-Ramsey coloring first obtained by Todorcevic in \cite{stevo-acta}.
Wark modified the norm in \cite{hugh} introducing $\ell_2$ norm and obtaining Hilbert generated
examples (further modifications which we do not consider here were even
 reflexive, or in \cite{hugh-sm}  even with stronger properties).  Such spaces apparently have high potential
 for carrying strong anti-Ramsey features of the metric on the unit sphere. This was investigated in \cite{pk-kamil}. 
 In fact they are generated as Banach spaces by a set $\{x_\alpha: \alpha<\omega_1\}$
 (characteristic functions of singletons)
 which is hyperlateral, in fact 
 $\|x_\alpha-x_\beta\|$ is $\sqrt{2}$ or $1$ depending on the value of the coloring
 on the pair $\{\alpha,\beta\}$. Propositions
 \ref{zfc-sw}, \ref{hilbert-gen-ch}, \ref{hilbert-gen-ma} summarize
 the properties of the spheres of these spaces in {\sf ZFC}, under {\sf CH}
 and under {\sf MA} and the negation of {\sf CH} respectively.
 It turns out that under the latter hypothesis the unit spheres are
 almost dichotomous despite being generated by a hyperlateral set.

 Recall that for a family $\A$ of finite subsets of $\omega_1$ which is closed under taking subsets and
 contains all singletons we can define a norm on the space $c_{00}(\omega_1)$ of finitely supported elements
 of $\R^{\omega_1}$ 
by
$$\|x\|_\A=\sup_{A\in \A}\sqrt{\sum_{\alpha\in A}x(\alpha)^2}.$$
The completion of $c_{00}(\omega_1)$ with respect to this norm
is denoted $(\X_\A, \|\ \|_\A)$.  See \cite{pk-kamil} for more details. Such spaces are always Hilbert generated
 (see Proposition 5 of \cite{pk-kamil}).
If $c:[\omega_1]^2\rightarrow\{0,1\}$, then we
can consider the family $\A_c$ of all finite $0$-monochromatic sets together with the singletons. 

\begin{lemma}\label{squares-sw} Let $\A$ be
a family of finite subsets of $\omega_1$ which is closed
under taking subsets and  contains all singletons. Then for every $x, x'\in \X_\A$ with disjoint supports we have 
$$\|x-x'\|^2_\A\leq \|x\|^2_\A+\|x'\|^2_\A.$$
\end{lemma}
\begin{proof}
For each $A, B\in \A$ we have ${\sum\{x(\alpha)^2}: {\alpha\in A\cap(supp(x))}\}\leq\|x\|_\A^2$
and ${\sum\{x'(\alpha)^2}: {\alpha\in B\cap(supp(x'))}\}\leq\|x'\|_\A^2$ and so
$${\sum\{\big((x-x')(\alpha)\big)^2}: {\alpha\in A}\}\leq$$
$$\leq{\sum\{x(\alpha)^2}: {\alpha\in A\cap supp(x)}\}+{\sum\{x'(\alpha)^2}: {\alpha\in A\cap supp(x')}\}\leq$$
$$\leq\|x\|_\A^2+\|x'\|_\A^2.$$
As $A\in \A$ was arbitrary we conclude the lemma.
\end{proof}

\begin{proposition}\label{sw-unions} Let $\A$ be
a family of finite subsets of $\omega_1$ which is closed
under taking subsets and  contains all singletons. 
Let $\X$ be equal to $(\X_{\A}, \|\ \|_{\A})$. For every $\varepsilon>0$
 the unit sphere $S_\X$
is  the union of countably many sets
of diameters not bigger than $\sqrt2+\varepsilon$. So $(\sqrt2, 2]\subseteq\Sigma(S_\X)$.
\end{proposition} 
\begin{proof}
Because of Lemma \ref{squares-sw} the proof of this proposition is  very similar to  the proof of
Proposition \ref{lp-union}, so we leave it to the reader.
 \end{proof}

\begin{proposition}[\cite{pk-kamil}]\label{zfc-sw}  Suppose that
 $c:[\omega_1]^2\rightarrow\{0,1\}$ does not admit  uncountable monochromatic sets.
The Banach space $\X$ equal to $(\X_{\A_c}, \|\ \|_{\A_c})$ has the
following properties:
\begin{enumerate}
\item {\rm (\cite{pk-kamil})} $S_\X$ does admit an uncountable $(1+)$-separated set.
\item There is an uncountable $\Y\subseteq S_\X$ where no uncountable
 $\Z\subseteq \Y$  has diameter less than $\sqrt2$.
\end{enumerate}
Consequently 
\begin{itemize}
 \item $[0,1]\subseteq{\mathsf K}^+(S_{\X_{\A_c}})\subseteq[0,\sqrt2]$, 
 \item $(\sqrt2, 2]\subseteq \Sigma(S_{\X_{\A_c}})\subseteq [\sqrt2, 2]$.
\end{itemize}
\end{proposition}
\begin{proof}
(1) is  Proposition 19 of \cite{pk-kamil}.
To obtain (2) consider characteristic functions of singletons in $\omega_1$ and use
the hypothesis that the coloring has no uncountable monochromatic sets.
The second  part of the proposition follows from items (1) and (2) (as (2) implies
that the unit sphere can not be the union of countably many sets of diameters less than $\sqrt2$)
and from
Proposition \ref{sw-unions}.
\end{proof}

If $c$ is a coloring with special anti-Ramsey properties whose existence is implied
by the existence of a nonmeager subset of the reals of cardinality $\omega_1$,
then the spaces $(\X_{\A_c}, \|\ \|_{\A_c})$ are examples  very relevant to the subject matter of
this paper:

\begin{proposition}[\cite{pk-kamil}]\label{hilbert-gen-ch} Assume that  there is a nonmeager subset
of $\R$ of cardinality $\omega_1$ {\rm (}e.g. assume {\sf CH}{\rm )}. Then there
is a coloring $c:[\omega_1]^2\rightarrow\{0,1\}$ with no uncountable monochromatic sets such that
the unit sphere of the Banach space $(\X_{\A_c}, \|\ \|_{\A_c})$ is not almost
dichotomous and is  hyperlateral: 
for every $\delta>0$ there is $\varepsilon>0$
such that for every uncountable $(1-\varepsilon)$-separated set $\{x_\alpha: \alpha<\omega_1\}\subseteq S_{\X_{\A_c}}$
 there are $\alpha<\beta<\omega_1$ such that $\|x_\alpha-x_\beta\|_{\A_c}>\sqrt2-\delta$ and there are
 $\alpha<\beta<\omega_1$ such that $\|x_\alpha-x_\beta\|_{\A_c}<1+\delta$.
 We have 
 \begin{itemize}
 \item ${\mathsf K}^+(S_{\X_{\A_c}})=[0,1]$, 
 \item $(\sqrt2, 2]\subseteq \Sigma(S_{\X_{\A_c}})\subseteq [\sqrt2, 2].$
\end{itemize}
 \end{proposition}
 \begin{proof} The property stated 
 in the proposition (responsible for $S_{\X_{\A_c}}$ being hyperlateral above) is 
 stated in Proposition 23 of \cite{pk-kamil}.  The remaining parts follow from
 Propositions \ref{zfc-sw} and  \ref{riesz}.
 \end{proof}

Even strong {\sf ZFC} anti-Ramsey properties of colorings of pairs of $\omega_1$ due to Todorcevic
and exploited in \cite{shelah, ss, hugh, hugh-sm} are not sufficient to maintain in {\sf ZFC}
the properties of the spaces from Theorem \ref{hilbert-gen-ch} as we have the following:

\begin{proposition}[\cite{pk-kamil}]\label{hilbert-gen-ma} Assume {\sf MA} and the negation of {\sf CH}.
Let $c:[\omega_1]^2\rightarrow\{0,1\}$ be any coloring without uncountable monochromatic sets. 
Then the unit sphere of every space of the form $(X_{\A_c}, \|\ \|_{\A_c})$ is almost dichotomous,
admits an uncountable equilateral set and
\begin{itemize}
\item  $[0,\sqrt2)\subseteq{\mathsf K}^+(S_{\X_{\A_c}})\subseteq[0,\sqrt2]$, 
\item  $(\sqrt2, 2]\subseteq \Sigma(S_{\X_{\A_c}})\subseteq [\sqrt2, 2]$.
\end{itemize}
 \end{proposition}
\begin{proof} It is proved in Proposition 19 of \cite{pk-kamil} that under
the above hypothesis the sphere $S_{\X_{\A_c}}$ 
admits an uncountable $\sqrt2$-equilateral
set, so $[0,\sqrt2)\subseteq{\mathsf K}^+(S_{\X_{\A_c}})$. 
The remaining parts are consequences of Proposition \ref{zfc-sw}.
\end{proof}

One should add that refinements of the norms $\|\ \|_\A$ were introduced in
\cite{ss} and were successfully exploited in \cite{hugh, hugh-sm, chang}. 
Only some   partial results concerning this class of spaces in the present context can be found in \cite{pk-kamil}.

\section{Banach spaces of the form $C(K)$ and $C_0(K)$}

In this section we look at Banach spaces of continuous functions on
compact Hausdorff spaces and at Banach spaces of continuous functions vanishing at infinity on locally compact Hausdorff
spaces both with the supremum norm $\|\ \|_\infty$. The situation in the first case has been 
well investigated in the literature and we summarize it and express it in our
language in Theorem \ref{thm-merc} and \ref{hyper-ck} with some
refinements of results of \cite{equi} in Theorem \ref{hyper-ck}. In particular, under {\sf MA} and the negation of {\sf CH}
the spheres of $C(K)$ spaces are trivially dichotomous because they (often nontrivially) admit uncountable $2$-equilateral sets. 
But a strong hyperlateral $C(K)$ space does exist consistently (Theorem \ref{hyper-ck}).
Then we look at the simplest example of locally compact noncompact nonmetrizable space $K$ that is, the uncountable discrete
space whose $C_0(K)$ is just $c_0(\kappa)$ for an uncountable cardinal $\kappa$.
It is well known that this Banach space already displays different behavior than
$C(K)$ spaces, among others it does not admit an uncountable $(1+\varepsilon)$-separated set
in {\sf ZFC}. We prove that the unit  sphere of $c_0(\kappa)$ is dichotomous (Proposition \ref{c0-dich})
but it admits an uncountable subset which is not dichotomous (Proposition \ref{c0-notherdich}).
Then we analyse a general case of locally compact noncompact nonmetrizable $K$. First,
we present three general results in {\sf ZFC}, first two  on the existence
of uncountable equilateral sets ({Propositions \ref{C0K-total} and \ref{C0K-equi}) and 
in the third one we show that a nonseparable $C_0(K)$ always admits an uncountable
 $(1+)$-separated set (Proposition \ref{lc-1+}). Then we provide two {\sf ZFC} examples,
the first of \cite{ad-kottman} shows that the behavior of separated sets can
be similar to that in $c_0(2^\omega)$ already in separable locally compact spaces 
(Proposition \ref{C0K-strange}),
and the second example is of a Banach space of density $2^\omega$ with a
non-dichotomous unit sphere (Proposition \ref{C0K-counter}). 
The last result (Proposition \ref{C0K-small}) is a dichotomy 
for $C_0(K)$s for locally compact $K$s of the weight less than $2^\omega$ under {\sf MA} and the negation
of {\sf CH}: It turns out that  under this hypothesis the unit spheres of $C_0(K)$s
behave either as the unit spheres of $C(K)$s for compact $K$s, or as the unit sphere 
of $c_0(\kappa)$.

\begin{theorem}\label{thm-merc} Suppose that $K$ is compact and Hausdorff.
\begin {enumerate}
\item (Mercourakis,  Vassiliadis  \cite{mer-ck}) If $K$ is  nonmetrizable and has one of the following properties:
\begin{itemize}
\item $K$ is totally disconnected, 
\item $K$ is not hereditarily
separable, 
\item $K$ is not hereditarily Lindel\"of, 
\item $K$ carries  a Radon measure of uncountable Maharam type,
\end{itemize}
then  the unit sphere of 
$C(K)$ admits an uncountable $2$-equilateral set.
In particular, such unit spheres are dichotomous with ${\mathsf K}^+(S_{C(K)})=[0, 2)$ and
${\Sigma}(S_{C(K)})=\{2\}$.
\item (Kania, Kochanek \cite{tt}) If $K$ is nonmetrizable space, then
the unit sphere $S_{C(K)}$ admits an uncountable $(1+)$-separated set, 
i.e., $[0,1]\subseteq {\mathsf K}^+(S_{C(K)})$.
\item {\sf MA} and the negation of {\sf CH} imply that the unit sphere of 
every nonseparable Banach space of the form $C(K)$ admits an uncountable $2$-equilateral set.
In particular, such unit spheres are dichotomous with ${\mathsf K}^+(S_{C(K)})=[0, 2)$ and
${\Sigma}(S_{C(K)})=\{2\}$.
\end{enumerate}
\end{theorem}

Recall that a strongly Luzin set (see Section 6 of \cite{stevo-problems})
is  an uncountable subset $L$ of $\R$ such that for every $n\in \N$ any
uncountable subset of $L^n$ consisting of disjoint $n$-tuples (i.e., whose ranges are 
pairwise disjoint and all coordinates different) is dense in some open subset of $\R^n$.

\begin{proposition}\label{hyper-ck} Suppose that there 
exists a strongly Luzin subset of $[0,1]$ of cardinality $\omega_1$ (e.g. assume {\sf CH}).
Then there is a separable compact Hausdorff space
$K$ (so $C(K)$ is isometric to a subspace of $\ell_\infty$) such that $C(K)$ does not admit an uncountable 
equilateral set,  such that $S_{C(K)}$ is hyperlateral (so
is not almost dichotomous): for every $\varepsilon>0$
 all uncountable sets of $S_{C(K)}$ which are $(1-\varepsilon)$-separated
contain $x, x', y, y'$ such that $\|x-x'\|\leq1+3\varepsilon$ and $\|y-y'\|\geq2-8\varepsilon$.
In particular, ${\mathsf K}^+(S_{C(K)})=[0, 1]$ and ${\Sigma}(S_{C(K)})=\{2\}$.
\end{proposition}
\begin{proof} 
This is the space of \cite{equi} which is constructed there by forcing rather
than from an additional set-theoretic hypothesis.
The first part of the proposition concerning equilateral sets is one of the main results of \cite{equi},
namely Theorem 3.3.
So what we need to verify is ${\mathsf K}^+(S_{C(K)})=[0, 1]$ and
${\Sigma}(S_{C(K)})=\{2\}$.
Looking at the proof of Theorem 3.3 in \cite{equi} we see that we start with
an arbitrary set $\{e_\alpha:\alpha<\omega_1\}$ in the unit ball of the $C(K)$, for
each $\alpha<\omega_1$
we approximate  $e_\alpha$ with some  $g_\alpha'$ and then we produce
$g_\alpha$s which satisfy $g_\alpha-g_{\beta}=g_\alpha'-g_{\beta}'$ for any 
$\alpha, \beta<\omega_1$. Moreover $g_\alpha$s are of the form
for which we can apply Proposition 3.2. of \cite{equi} 
obtaining distinct  $\alpha, \beta, \alpha', \beta'<\omega_1$ (by this Proposition 3.2. of \cite{equi})
 such that 
$$\|g_{\alpha'}-g_{\beta'}\|=2\|g_\alpha-g_\beta\|.\leqno (1)$$
We repeat the same argument (by the same reasoning $\delta$ in that proof can be assumed to 
be bigger than $0$)
but additionally assume 
that $\{e_\alpha:\alpha<\omega_1\}$ is $(1-\varepsilon)$-separated subset of the unit
sphere and we additionally require $\|g_\alpha'-e_\alpha\|<\varepsilon$.
By the above for every $\xi,\eta<\omega_1$ we have

$$\|g_\xi-g_\eta\|-2\varepsilon\leq\|e_\xi-e_\eta\|\leq \|g_\xi-g_\eta\|+2\varepsilon.\leqno(2)$$
Considering 
$\alpha, \beta, \alpha', \beta'<\omega_1$ as in (1),
from (2) we get that $\|g_\alpha-g_\beta\|\geq 1-3\varepsilon$ since
 $\{e_\alpha:\alpha<\omega_1\}$ is $(1-\varepsilon)$-separated 
 and so  by (1)
 $$\|g_{\alpha'}-g_{\beta'}\|\geq 2-6\varepsilon.\leqno (3)$$ 
 By (2) $\|g_{\alpha'}-g_{\beta'}\|-2\varepsilon\leq 2$ so
 by (1)  
 $$\|g_{\alpha}-g_{\beta}\|\leq 1+\varepsilon.\leqno (4)$$
 Applying (2) to the items (3) and (4) we obtain
  $\|e_{\alpha'}-e_{\beta'}\|\geq 2-8\varepsilon$
  and $\|e_{\alpha}-e_{\beta}\|\leq 1+3\varepsilon$, as required in the proposition.
 
 It remains to show that the space of \cite{equi} exists under the assumption of
 the existence of a strongly Luzin set of cardinality $\omega_1$. We need to show that
 under this hypothesis there is  $\{r_\xi:\xi<\omega_1\}\subseteq[0,1]$
 and continuous $f_\xi: [0,1]\setminus\{r_\xi\}\rightarrow [0,1]$ for $\xi<\omega_1$
 which are anti-Ramsey in the sense of Definition 3.1 of \cite{equi}. 
 
 First note that  a strongly Luzin set $L$ of cardinality $\omega_1$
 has a formally stronger property that the one from its definition,
 namely for every $n\in \N$ any
uncountable subset $X$ of $L^n$ consisting of disjoint $n$-tuples (i.e., whose ranges are 
pairwise disjoint and all coordinates different) is $\omega_1$-dense in some open subset $U$ of $\R^n$
 i.e., such that $X\cap V$ is uncountable for each open nonempty $V\subseteq U$.
To prove this  consider the countable family 
 $\I$ of all  products $I_1\times\dots \times I_n$ of open intervals with rational end-points 
 satisfying $|(I_1\times\dots \times I_n)\cap X|\leq\omega$.
 Then $X\setminus \bigcup\I$ is an uncountable subset of $X$ consisting
  of disjoint $n$-tuples  and so somewhere dense since $L$ is  strongly Luzin.
 Hence the closure of $X\setminus \bigcup\I$ contains some product $J_1\times\dots \times J_n$ of
 open intervals.
 As $(J_1\times\dots \times J_n)\cap (I_1\times\dots \times I_n)=\emptyset$ for all $(I_1\times\dots \times I_n)\in \I$,
  we conclude that
 $X$ is $\omega_1$-dense in $J_1\times\dots \times J_n$, as required.
 
 As $\{r_\xi:\xi<\omega_1\}$ we take the irrational elements of the strongly Luzin set of cardinality $\omega_1$.
 It is clear that they form a strongly Luzin  set as well. 
 
 We choose open subintervals $K_s$ of $[0,1]$ with rational 
 endpoints for $s\in \N^{<\omega}=\bigcup_{n\in \N} \N^n$
 such that $K_\emptyset=(0,1)$, $K_s\subseteq K_t$ whenever $t\subseteq s$ for $s, t\in \N^{<\omega}$
 and such that every irrational element of $[0,1]$ is in exactly one interval $K_s$
 for $s\in \N^n$ for every $n\in \N$ (in particular,
 $\{K_s: s\in \N^n\}$ is pairwise disjoint for a fixed $n\in \N$), while
 every rational in $[0,1]$ appears just once as an endpoint of some interval $K_s$. We also
 require that $diam(K_s)<1/|s|$ for $s\not=\emptyset$.
 In particular for each irrational $r\in(0,1)$ there is a unique $x\in \{0,1\}^\N$ such that
 $r\in K_{x|n}$ for each $n\in \N$.

 Call a partial function from $[0,1]$ into itself a nice function
 if it continuous on its domain and is the finite sum of affine functions (possibly constant)
 defined on intervals with rational endpoints where the inclinations are rational as well.
 It follows that there are countably many nice functions.
 
 By recursion for $n\in \N$ we define enumerations $\phi_s$ with domains $\N$ for $s\in \N^n$ with infinitely  many
 repetitions of some 
 subsets of nice functions defined on $(0,1)\setminus K_s$. 
 For $n=0$ the enumeration $\phi_\emptyset: \N\rightarrow \{\emptyset\}$
 is constantly the empty function.  If $n>0$ the enumeration $\phi_s$ 
 enumerates with infinitely many repetitions all nice functions $(0,1)\setminus K_s$ which include (i.e., continuously extend)
 the nice function  $\phi_{s|(n-1)}(s(n-1))$.

 Finally we define   $f_\xi: [0,1]\setminus \{r_\xi\}\rightarrow [0,1]$
 as  $f_\xi=\bigcup_{n>0} \phi_{s_n}(s_{n+1}(n))$, where $(s_n)_{n\in \N}$
 is such that $s_n\in \N^n$ and $r_\xi\in K_{s_n}$. 
 Note that $f_\xi$ is continuous on  $[0,1]\setminus \{r_\xi\}$
 since $\phi_{s_n}(s_{n+1}(n))\subseteq \phi_{s_{n+1}}(s_{n+2}(n+1))$
 as all values of $\phi_{s_{n+1}}$ are continuous extensions of 
 $\phi_{s_{n+1}|n}(s_{n+1}(n))=\phi_{s_n}(s_{n+1}(n))$. It follows that
 for each $i\in \N$, $s\in \{0,1\}^{<\omega}$ and $\xi<\omega_1$ we have 
 $$r_\xi\in K_{s^\frown i}\ \ \Rightarrow\ \ f_\xi|((0,1)\setminus K_s)=\phi_{s}(i).\leqno (5)$$
 
 It remains to show that $f_\xi$s are anti-Ramsey. 
 Fix $m\in \N$, pairwise disjoint finite $(F_\alpha)_{\alpha<\omega_1}$ in $\omega_1$
 where $F_\alpha=\{r_{\xi^\alpha_1}, \dots, r_{\xi^\alpha_m}\}$, pairwise
 disjoint subintervals $I_1, \dots, I_m$ of $[0,1]$ with rational endpoints
 such that $r_{\xi^\alpha_i}\in I_i$ for each $\alpha<\omega_1$ and $1\leq i\leq m$
 and rationals $\{q_1, \dots, q_m\}$ and $\{q_1', \dots, q_m'\}$.

 By the strong Luzin property $\{(r_{\xi^\alpha_1}, \dots, r_{\xi^\alpha_m}): \alpha<\omega_1\}$
 is $\omega_1$-dense in  some  $I_{1}'\times\dots \times I_{m}'$ with $I_i'\subseteq I_i$ for $1\leq i\leq m$.
 By the construction of $K_s$s there is $n\in \N$ and
 $s_1, \dots s_m\in \N^n$ such that $K_{s_i}\subseteq I_i'$ for $1\leq i\leq m$, in particular
  $\{(r_{\xi^\alpha_1}, \dots, r_{\xi^\alpha_m}): \alpha<\omega_1\}$
 is $\omega_1$-dense in $K_{s_1}\times\dots \times K_{s_m}$.

Using the fact that $\phi_{s_i}$ are with infinitely many repetitions find $k_1, \dots, k_m$ and $l_1, \dots, l_m$ such that
$k_i\not=l_i$, $\phi_{s_i}(k_i)=\phi_{s_i}(l_i)$ for  $1\leq i\leq m$ and the intervals 
$K_{s_i^\frown k_i}$ and $K_{s_i^\frown l_i}$ have all distinct end-points.
By  (5) this guarantees that
$$f_{\xi_i}|((0,1)\setminus K_{s_i})=f_{\eta_i}|((0,1)\setminus K_{s_i})=\phi_{s_i}(k_i)=\phi_{s_i}(l_i)$$
whenever $r_{\xi_i}\in K_{s_i^\frown k_i}$ and $r_{\eta_i}\in K_{s_i^\frown l_i}$
and that there are continuous nice extensions $g_i:[0,1]\rightarrow[0,1]$ 
of $\phi_{s_i}(k_i)=\phi_{s_i}(l_i)$ such that
$$g_i|K_{s_i^\frown k_i}=q_i, \ g_i|K_{s_i^\frown l_i}=q_i'.\leqno (6)$$
 Let  $k_1', \dots, k_m'$ and $l_1', \dots, l_m'$ for $1\leq i\leq m$ be such that
 $$\phi_{s_i^\frown k_i}(k_i')= g_i|((0,1)\setminus K_{s_i^\frown k_i}), \ \ 
 \phi_{s_i^\frown l_i}(l_i')= g_i|((0,1)\setminus K_{s_i^\frown l_i}).\leqno (7)$$
 The existence of such integers follows from
 the fact that $\phi_s$'s enumerate all nice extensions of $\phi_{s'}$ where
 $s'$ is the initial fragment of $s$ shorter by one. 
 
Using the $\omega_1$-density of   $\{(r_{\xi^\alpha_1}, \dots, r_{\xi^\alpha_m}): \alpha<\omega_1\}$
in $K_{s_1}\times\dots \times K_{s_m}$ find $\alpha<\beta<\omega_1$
such that 
$$(r_{\xi^\alpha_1}, \dots, r_{\xi^\alpha_m})\in K_{s_1^\frown l_1^\frown l_1'}\times\dots\times
K_{s_m^\frown l_m^\frown l_m'},$$
 $$(r_{\xi^\beta_1}, \dots, r_{\xi^\beta_m})\in K_{s_1^\frown k_1^\frown k_1'}\times\dots\times
K_{s_m^\frown k_m^\frown k_m'}.$$
For $1\leq i\leq m$ put $J_i^\alpha= K_{s_i^\frown l_i}$ and $J_i^\beta= K_{s_i^\frown k_i}$.
We obtain
\begin{itemize}
\item $J_i^\alpha, J_i^\beta\subseteq I_i$,
\item  $r_{\xi^\alpha_i}\in J_i^\alpha$, $r_{\xi^\beta_i}\in J_i^\alpha$,
\item $f_{\xi^\alpha_i}([0,1]\setminus (J_i^\alpha\cup J_i^\beta))
= f_{\xi^\beta_i}([0,1]\setminus (J_i^\alpha\cup J_i^\beta))$ by (7),
\item $f_{\xi^\alpha_i}|J_i^\beta=q_i$ by (6) and (7),
\item $f_{\xi^\beta_i}|J_i^\alpha=q_i'$ by (6) and (7),
\end{itemize}
for all $1\leq i\leq m$ as required in Definition 3.1 of \cite{equi}
 which completes the proof that a strongly Luzin set of cardinality $\omega_1$ gives the space
 of \cite{equi}.
\end{proof}

Now we move to locally compact Hausdorff spaces $K$ and the unit spheres of
the Banach spaces $C_0(K)$.  A fundamental example
of a nonmetrizable  locally compact noncompact space is the uncountable  discrete space. Its $C_0(K)$ is $c_0(\kappa)$
for some uncountable cardinal $\kappa$.
Already in the paper \cite{elton-odell} of Elton and Odell it was noted that
$c_0(\omega_1)$ is an interesting example because it does not admit an uncountable
$(1+\varepsilon)$-equilateral set for any $\varepsilon>0$.  As the discrete space is totally disconnected,
this already shows a different behavior than in the compact case (cf.  the first item
of Theorem \ref{thm-merc} (1)). 
We prove below that
in fact for every $\varepsilon>0$ 
its unit sphere is the countable union of sets of diameter not bigger than $1+\varepsilon$ and
so the sphere is dichotomous (as the space admits an uncountable
$(1+)$-separated set, which is elementary but also follows from our
Proposition \ref{lc-1+}). The method follows the paper of Erd\"os and Preiss \cite{erdos-p}
and uses the Hewitt Marczewski Pondiczery Theorem \ref{hmp}.

\begin{lemma}\label{c0-delta} Let $\kappa$ be  cardinal. 
Suppose that $x, y\in S_{c_0(\kappa)}$ satisfy $x|a=y|a$,  where $a=supp(x)\cap supp(y)$.
Then $$\|x-y\|_\infty\leq1.$$
\end{lemma}

\begin{proposition}\label{c0-dich} For
every $\varepsilon>0$ the unit sphere
$S_{c_0(2^\omega)}$ is the union of countably many sets of diameters
not bigger than $1+\varepsilon$.  Consequently $S_{c_0(\kappa)}$
is dichotomous for every $\kappa\leq2^\omega$.
\end{proposition}
\begin{proof} Fix $\varepsilon>0$. It is enough to show that
the set $\Y$ of all elements of the unit sphere of $c_0(2^\omega)$ with finite supports and assuming
rational values
is the union of countably many sets of diameter not bigger than $1$.
Let $D=\{d_n: n\in \N\}\subseteq \Q^{(2^\omega)}$ be a dense subset of  $\Q^{(2^\omega)}$ 
where $\Q$ is considered with discrete topology
which exists by the Hewitt-Marczewski-Pondiczery theorem.
It follows that for each element $y\in \Y$ there is $n_y\in \N$ such that $d_{n_y}|supp(y)=y|supp(y)$.
We claim that $\Y_n=\{y\in \Y: n_y=n\}$ has diameter not bigger than $1$.
This follows from Lemma \ref{c0-delta}. The existence of an uncountable 
$(1+)$-separated set in $c_0(\omega_1)$ follows from the next Proposition \ref{c0-notherdich}.
\end{proof}

\begin{proposition}\label{c0-notherdich} In $S_{c_0(\omega_1)}$ there is an uncountable $\X$ which 
is $(1+)$-separated and there is an uncountable $\Y$ without uncountable $\Y'\subseteq \Y$
which is $(1+)$-separated  and without an uncountable $\Y'\subseteq \Y$ which has
the diameter not bigger than $1$.
\end{proposition}
\begin{proof} $\X=\{x_\alpha: \alpha<\omega_1\}\subseteq S_{c_0(\omega_1)}$, where
$x_\alpha(\alpha)=1$, $x_\alpha(\beta)<0$ for all $\beta<\alpha$ and 
 $x_\alpha(\beta)=0$ for all $\beta>\alpha$.
 
 Let $c: [\omega_1]^2\rightarrow \{0,1\}$ be the coloring of Theorem \ref{sierpinski1}.
 $\Y=\{y_\alpha: \alpha<\omega_1\}\subseteq S_{c_0(\omega_1)}$, where
$y_\alpha(\alpha)=1$, $y_\alpha(\beta)<0$ if  $\beta<\alpha$ and $c(\{\beta, \alpha\})=1$
and  $y_\alpha(\beta)=0$ for all $\beta>\alpha$ and $\beta<\alpha$ if $c(\{\beta, \alpha\})=0$
\end{proof}

The above proposition shows that already in such a reasonable space
as $c_0(\omega_1)$ with the unit sphere dichotomous we can have uncountable subsets of
the unit sphere which are not dichotomous in {\sf ZFC}.

To consider the general case of locally compact spaces we need the following:

\begin{lemma}\label{hL} Suppose that $K$ is locally compact Hausdorff nonmetrizable
space which does not admit a nonmetrizable compact subspace.
Then there are open $U_\alpha\subseteq K$ with compact closures and $x_\alpha\in K$ for $\alpha<\omega_1$
such that $x_\alpha\in U_\alpha\setminus\bigcup_{\beta<\alpha}U_\beta$.
\end{lemma}
\begin{proof}
 Suppose that $K$ is locally compact Hausdorff nonmetrizable
space which does not admit a nonmetrizable compact subspace.
If we can cover $K$ by countably many open $U$s with $\overline U$ compact,
then as each such $U$ is metrizable with $\overline U$ compact, it has a countable
basis. So then $K$ has a countable basis, which contradicts the nonmetrizability of $K$.
Hence such cover can not exist. 

So construct $U_\alpha, x_\alpha$ by transfinite recursion demanding the closure of $U_\alpha$
to be compact and $x_\alpha\in U_\alpha$: having constructed
$U_\beta$ for $\beta<\alpha$, we know that $\bigcup_{\beta<\alpha}U_\beta$
does not cover $K$, so there is $x_\alpha\in K\setminus \bigcup_{\beta<\alpha}U_\beta$.
Now choose $U_\alpha$ by the local compactness at $x_\alpha$.
\end{proof}

The following result generalizes the theorem of
Kania and Kochanek of \cite{tt} (our Theorem \ref{thm-merc} (2)).

\begin{proposition}\label{lc-1+} Suppose that $K$ is a locally compact nonmetrizable  Hausdorff space.
Then $C_0(K)$ admits an uncountable $(1+)$-separated set.
\end{proposition}
\begin{proof}

Note (*): given an open $U\subseteq K$ and 
a countable $X=\{x_n: n\in \N\}\subseteq K\setminus \overline U$ we can build
$f\in C_0(K)$ such that $f(x)>0$ for all $x\in X$, $\|f\|_\infty\leq 1$ and $f|U=0$. Just put $f=\sum_{n\in \N}f_n$
where $f_n|\overline U=0$, $f_n(x_n)>0$ and $0\leq f_n(x)\leq 1/2^{n+1}$ for each $x\in K$ and each $n\in \N$
  so that the sequence of partial sums is Cauchy.
  
  Consider the first case of $K$ nonseparable. Then we can build
  $\{x_\alpha: \alpha<\omega_1\}\subseteq K$ such that for each $\alpha<\omega_1$
  there is an open $U_\alpha$ with compact closure
  such that $x_\alpha\in U_\alpha$ and 
  $\{x_\beta: \beta<\alpha\}\cap \overline{U_\alpha}=\emptyset$ for each $\alpha<\omega_1$.
  So using (*) build $f_\alpha\in C_0(K)$ such that $f_\alpha(x_\beta)>0$ for each $\beta<\alpha$
  and $f_\alpha|\overline U_\alpha=0$. 
  Now consider $g_\alpha\in C_0(K)$ of norm one with support included in $U_\alpha$
  such that $g_\alpha(x_\alpha)=1$.
   We obtain that $\{g_\alpha-f_\alpha: \alpha<\omega_1\}\subseteq S_{C_0(K)}$
  is $(1+)$-separated as 
  $$((g_\alpha-f_\alpha)-(g_\beta-f_\beta))(x_\beta)=-f_\alpha(x_\beta)-g_\beta(x_\beta)=-f_\alpha(x_\beta)-1<-1$$
  for any $\beta<\alpha<\omega_1$.
  
Now consider the second case of $K$ separable with a countable dense $D=\{d_n: n\in \N\}$.
If $K$ admits an open $U$ with compact closure which is nonmetrizable, then 
$C(\overline U)$ admits an uncountable $(1+)$-separated set by the result of
Kania and Kochanek (Theorem \ref{thm-merc} (2)). So by the Tietze theorem for locally compact spaces
we may assume that
if $U$ is open with $\overline U$ compact, then $\overline U$ is metrizable.
Note that this implies that $K$ contains no compact nonmetrizable space as the local compactness implies
that every compact subspace of $K$ is included in an open set with compact closure. Hence we may assume that the hypothesis
of Lemma \ref{hL} is satisfied.

It follows from Lemma \ref{hL} that there are open $U_\alpha, V_\alpha$ for
$\alpha<\omega_1$ with $\overline U_\alpha$ compact and $x_\alpha\in V_\alpha\subseteq\overline V_\alpha
\subseteq U_\alpha$
such that $x_\alpha\not\in \bigcup_{\beta<\alpha}U_\beta$. 
Using (*) build $f_\alpha\in C_0(K)$ of norm one such that $f_\alpha(d_n)>0$
for each $d_n\not\in \overline V_\alpha$ and $f_\alpha|V_\alpha=0$.  Let
$g_\alpha\in C_0(K)$ be of norm one with support included in $V_\alpha$ such that $g_\alpha|W_\alpha=1$,
where $W_\alpha$ is an open neighborhood of $x_\alpha$ whose closure is  included in $V_\alpha$.
Note that if $\beta<\alpha<\omega_1$  we have that 
$$x_\alpha\in W_\alpha\setminus U_\beta
\subseteq W_\alpha\setminus \overline V_\beta, $$
 so for all $n\in \N$ such that $d_n$ is in the nonempty open $W_\alpha\setminus \overline V_\beta$
we have 
 $$g_\beta(d_n)=0\ \hbox{and}\ f_\beta(d_n)>0\ \hbox{and}\ g_\alpha(d_n)=1.$$
 So we obtain that $\{g_\alpha-f_\alpha: \alpha<\omega_1\}\subseteq S_{C_0(K)}$
  is $(1+)$-separated as 
  $$((g_\alpha-f_\alpha)-(g_\beta-f_\beta))(d_n)=(g_\alpha(d_n)-0)-(0-f_\beta(d_n))=1+f_\beta(d_n)>1$$
  for any $\beta<\alpha<\omega_1$ and $n\in \N$ such that $d_n\in W_\alpha\setminus \overline V_\beta$.
\end{proof}

We also have a result corresponding to the theorem of Mercourakis and Vassiliadis
(Theorem \ref{thm-merc} (1) the first item):

\begin{proposition}\label{C0K-total} If $K$ is locally compact, totally disconnected and nonmetrizable,
then $S_{C_0(K)}$ admits an uncountable equilateral set.
\end{proposition}
\begin{proof} By Theorem \ref{thm-merc} (1) we may assume that every compact subspace of $K$ is metrizable.
By Lemma \ref{hL} (since $U_\alpha$s can be assumed to be clopen)
there are uncountably many distinct open compact subsets
of $K$. Characteristic functions of such sets are continuous and the distance between them  is $1$, so
we obtain an uncountable $1$-equilateral subset of the unit sphere $S_{C_0(K)}$.
\end{proof}

Note that unlike in the compact case (Theorem \ref{thm-merc} (1) the first item), we cannot
count on uncountable $2$-equilateral set 
in the unit sphere in the locally compact case because it does not exist even
in the unit sphere of $c_0(\omega_1)$ by Proposition \ref{c0-dich}.
The question of the existence of uncountable equilateral sets in Banach spaces 
of the form $C_0(K)$ for $K$ locally compact, in fact, reduces to the analogous question for
compact spaces or is positively resolved:

\begin{proposition}\label{C0K-equi} Suppose that $K$ is locally compact Hausdorff nonmetrizable
space. Then either $K$ admits a nonmetrizable compact subspace or  ${C_0(K)}$ admits
an uncountable equilateral set.
\end{proposition}
\begin{proof}
Suppose that $K$ does not admit a nonmetrizable compact subspace.
So by Lemma \ref{hL} for each $\alpha<\omega_1$ there is open $U_\alpha$ with compact closure and 
$x_\alpha\in U_\alpha$ such that $x_\alpha\not\in \bigcup_{\beta<\alpha}U_\beta$ for every $\alpha<\omega_1$.
Consider $f_\alpha\in C_0(K)$ with supports included in $U_\alpha$
such that $0\leq f_\alpha(x)\leq 1$ for each $x\in K$ and $f_\alpha(x_\alpha)=1$. 
We have $\|f_\alpha-f_\beta\|\leq 1$, but also $f_\alpha(x_\alpha)-f_\beta(x_\alpha)=1-0=1$
for every $\beta<\alpha<\omega_1$. So $\{f_\alpha:\alpha<\omega_1\}$
forms an uncountable $1$-equilateral set.
\end{proof}

\begin{proposition}[\cite{ad-kottman}]\label{C0K-strange} There is a  totally disconnected 
 locally compact separable Hausdorff space  $K$ of weight $2^\omega$ (and any uncountable
 smaller weight)
such that ${\mathsf K}^+(S_{C_0(K)})=[0, 1]$ and
${\Sigma}(S_{C_0(K))}=(1, 2]$.  In particular, the
unit sphere of such $C_0(K)$  admits uncountable $1$-equilateral sets but
does not admit uncountable $r$-equilateral sets for any $r>1$.
\end{proposition}
\begin{proof} One needs to apply Proposition 48 of \cite{ad-kottman}
together with Proposition 27 of \cite{ad-kottman}. So the $\Psi$-space
induced by any $\R$-embeddable almost disjoint family has the properties of the proposition.
\end{proof}

\begin{proposition}\label{C0K-counter} There is a locally compact, Hausdorff 
scattered space of Cantor-Bendixson height $2$ and weight $2^\omega$ such that
for every $\varepsilon>0$ the  unit sphere $S_{C_0(K)}$ of $C_0(K)$ does not contain
an uncountable $(1+\varepsilon)$-separated  set and is not the union of countably
many sets of diameters less than $2$. Hence ${\mathsf K}^+(S_{C_0(K)})=[0,1]$
and $\Sigma(S_{C_0(K)})=\{2\}$ and so $S_{C_0(K)}$ is not almost dichotomous.

The sphere $S_{C_0(K)}$ admits an uncountable $1$-equilateral set, so it
is not hyperlateral.
\end{proposition}
\begin{proof}
The Banach space as in the proposition is isometric to a subspace  $\X$ of $\ell_\infty(\R)$ generated by $c_0(\R)$ and
vectors $1_{S_\alpha}$ for $\alpha<2^\omega$, where $S_\alpha$s are almost disjoint (i.e.,
with finite pairwise intersections). In fact, each
$S_\alpha$ 
is the set of all terms of a sequence of reals converging to $l_\alpha\in \R$ and not containing $l_\alpha$,
and $l_\alpha$s are distinct for $\alpha<2^\omega$. 
The appropriate values of $S_\alpha$s are constructed by transfinite recursion.
 
 Since the generators of $\X$ are closed under taking products in $\ell_\infty(\R)$, 
the Banach space 
$\X$ is a Banach subalgebra of $\ell_\infty(\R)$ and so by the Gelfand theorem it is 
isometric to a Banach space of the form
$C_0(K)$ for some locally compact $K$. The standard arguments like in the case
of almost disjoint families of subsets of $\N$ show that $K$
 is a scattered space of Cantor-Bendixson height $2$
  (with all points of $\R$ being isolated) and weight
   $2^\omega$ (see e.g., \cite{supermrowka}). Such versions of Johnson-Lindenstrauss spaces induced
 by almost disjoint families on uncountable sets rather than on $\N$
 were investigated  for example in \cite{wieslaw}.

To construct $S_\alpha$s (and $l_\alpha$s) first define a Cantor  net.  Consider closed intervals
$I_{s}$ for $s\in \bigcup_{n\in \N}\{0,1\}^n$, where $I_{\emptyset}=[0,1]$
and $I_{s^\frown0}$ and $I_{s^\frown1}$ are the leftside half and the rightside half
of $I_s$ for each $s\in \bigcup_{n\in \N}\{0,1\}^n$.  A Cantor net
will be a sequence $\mathcal T=( t_s: s\in \bigcup_{n\in \N}\{0,1\}^n)$ of  elements
of the interval $[0,1]$ such that $t_s\in int(I_s)$.  Note that for any $x\in \{0,1\}^\N$
the sequence $(t_{x|n})_{n\in \N}$ converges as it is Cauchy since 
$\{t_{x|n}: n\geq m\}\subseteq I_{s|m}$ for every $m\in \N$. Moreover
the limits of two sequences $(t_{x|n})_{n\in \N}$ and $(t_{x'|n})_{n\in \N}$
for distinct $x, x'\in \{0,1\}^\N$ are different if $x$ and $x'$ have infinitely many $0$s
and infinitely many $1$s.

Enumerate all Cantor nets $(\mathcal T_\alpha)_{\alpha<2^\omega}$.
Now suppose that $S_\beta$ and $l_\beta$ were constructed for all $\beta<\alpha$.  
To construct $S_\alpha$ and $l_\alpha$ consider
 $\mathcal T_\alpha=( t_s^\alpha: s\in \bigcup_{n\in \N}\{0,1\}^n)$. 
Let $S_\alpha=(t_{x|n}^\alpha)_{n\in \N}$ and $l_\alpha$ its limit  for $x\in \{0,1\}^\N$
such that the limit $l_\alpha$ is not among the points $l_\beta$ for $\beta<\alpha$
and not among the points $( t_s^\alpha: s\in \bigcup_{n\in \N}\{0,1\}^n)$
(in particular, $l_\alpha\not\in S_\alpha$)
and not among the end-points of the intervals $I_s$.
This completes the construction of $S_\alpha$s and $l_\alpha$s. Now we proceed to
proving the properties of $\X$.

First note that if $S_\X=\bigcup_{n\in \N}\X_n$, 
then there is
$n\in \N$ and $x, y\in \X_n$ such that $\|x-y\|_\infty=2$. We will find such two elements
in the set
$\{1_{S_\alpha}-1_{\{l_\alpha\}}:\alpha<2^\omega\}\subseteq S_\X$.
 To see this consider
a Cantor net $\mathcal T=( t_s: s\in \bigcup_{n\in \N}\{0,1\}^n)$
satisfying the following condition for every $n\in \N$: if there is $\beta<2^\omega$ and 
$1_{S_\beta}-1_{\{l_\beta\}}\in \X_n$ such that $l_\beta\in int(I_s)$ and $|s|=n$, then $t_s=l_\beta$
for some $\beta<2^\omega$ such that $1_{S_\beta}-1_{\{l_\beta\}}\in \X_n$. 
 Let $\alpha<2^\omega$ be such that
$\mathcal T=\mathcal T_\alpha$ and consider  $S_\alpha$.
Let $n_0\in \N$ be such that $1_{S_\alpha}-1_{\{l_\alpha\}}\in \X_{n_0}$.

By the construction $S_\alpha=(t_{x|n}^\alpha)_{n\in \N}$   for some $x\in \{0,1\}^\N$.
 Since $l_\alpha\in int(I_{x|n_0})$ 
we have that $t_{x|n_0}^\alpha=l_{\beta}$ for some $\beta<2^\omega$ and 
$1_{S_{\beta}}-1_{\{l_{\beta}\}}\in \X_{n_0}$. 

Moreover
$\alpha\not=\beta$ because $l_{\beta}\in S_\alpha$. Hence
$$\|(1_{S_\alpha}-1_{\{l_\alpha\}})-(1_{S_{\beta}}-1_{\{l_{\beta}\}})\|\geq
 |(1_{S_\alpha}(l_{\beta})-0)-(0-1_{\{l_{\beta}\}}(l_{\beta}))|=|1-(-1)|=2.$$

Now suppose that $(f_\alpha)_{\alpha<\omega_1}\subseteq S_\X$. Choose $\varepsilon>0$.
We aim at finding $\alpha<\beta<\omega_1$ such that $\|f_\alpha-f_\beta\|<1+\varepsilon$.
Find $g_\alpha$ with $\|g_\alpha\|\leq 1$ such that 
$$g_\alpha= h_\alpha+\sum_{i<j_\alpha} a^\alpha_i1_{S_{\xi^\alpha_i}\setminus F^\alpha_i},$$
where $h_\alpha\in c_{00}(\R)$ has all rational values, $a^\alpha_i\in \Q$, $j_\alpha\in \N$ and
$\xi^\alpha_i\in \omega_1$ and $supp(h_\alpha)$, $S_{\xi^\alpha_0}\setminus F^\alpha_0, \dots,
S_{\xi^\alpha_{j-1}}\setminus F^\alpha_{j-1}$ are pairwise disjoint and
$\|g_\alpha-f_\alpha\|<\varepsilon/2$ for all $\alpha<\omega_1$.
In particular, as $\|g_\alpha\|\leq 1$ we have that $|a_i^\alpha|\leq 1$
for all $\alpha<\omega_1$ and $i<j_\alpha$.
   
It is enough to find $\alpha<\beta<2^\omega$ such that $\|g_\alpha-g_\beta\|\leq 1$.
By passing to an uncountable subset we may assume that $j_\alpha=j\in \N$ for each $\alpha<\omega_1$.
By the $\Delta$-system Lemma, by passing to an uncountable subset we may assume that
$\{\{\xi^\alpha_0, \dots, \xi^\alpha_{j-1}\}: \alpha<\omega_1\}$ forms a $\Delta$-system
 with root $\Delta\subseteq2^\omega$ and for
 each element of $\Delta$ there is $i<j$ such that this element is $\xi^\alpha_i$
 with $a^\alpha_i=a^\beta_i$  and $F^\alpha_i=F^\beta_i$ for all $\alpha, \beta<\omega_1$. 
 By another application of the $\Delta$-system lemma we may assume that
the supports of $h_\alpha$s form a $\Delta$-system with root $\Delta'\subseteq\R$
 on which $h_\alpha$'s are equal. So if we consider the difference $g_\alpha-g_\beta$,
 the parts corresponding to the roots of $\Delta$-systems will cancel.
 Moreover subtracting this common part 
 $$h_\alpha|\Delta'+\sum_{\xi^\alpha_i\in \Delta} a^\alpha_i1_{S_{\xi^\alpha_i}\setminus F^\alpha_i}$$
from $g_\alpha$  does not increase the norm as $supp(h_\alpha)$, $S_{\xi^\alpha_0}\setminus F^\alpha_0, \dots,
S_{\xi^\alpha_{j-1}}\setminus F^\alpha_{j-1}$ are pairwise disjoint.
 This means that we may assume that $\{\{\xi^\alpha_0, \dots, \xi^\alpha_{j-1}\}: \alpha<\omega_1\}$
 is  pairwise disjoint family as well as the supports of $h_\alpha$s
 form a pairwise disjoint family.
 
For each $\alpha<\omega_1$ find pairwise disjoint open intervals $J^\alpha_i\subseteq \R$
 for $i<j$ with rational end-points
such that $l_{\xi^\alpha_i}\in J^\alpha_i$ and for all $k<j$ and $k\not=i$
we have
$$(\supp(h_\alpha)\cup S_{\xi^\alpha_k})\cap J^\alpha_i\subseteq\{l_{\xi^\alpha_i}\}.\leqno (*)$$
This can be done since $S_{\xi^\alpha_k}$ converges to $l_{\xi^\alpha_k}\not=l_{\xi^\alpha_i}$
and $supp(h_\alpha)$ is finite.

By passing to an uncountable set we may assume that
$a^\alpha_i=a_i$, $J^\alpha_i=J_i$ for all $\alpha<\omega_1$. Let
$J=\bigcup_{i<j}J_i$. Note that $supp(g_\alpha)\cap(\R\setminus J)$ 
is finite for each $\alpha<\omega_1$. So by passing to an uncountable
subset we may assume that the sets $supp(g_\alpha)\cap(\R\setminus J)$
form a $\Delta$-system and $g_\alpha$'s agree on its root and
so by Lemma \ref{c0-delta} the difference $g_\alpha-g_\beta$ when
restricted to $\R\setminus J$
has its norm not bigger than $1$ for every $\alpha<\beta<\omega_1$. It follows
that it is enough to find $\alpha<\beta<\omega_1$ such that 
$\|g_\alpha|J-g_\beta|J\|\leq 1$.

Now it is enough to find distinct $\alpha, \beta<\omega_1$
such that for every $i<j$ we have 
$$l_{\xi^\beta_i}\not\in S_{\xi^\alpha_i}\ \hbox{and}
\  l_{\xi^\alpha_i}\not\in S_{\xi^\beta_i}.\leqno (**)$$
This is because then, as $supp(g_\alpha)\cap J_i\subseteq S_{\xi^\alpha_i}\cup\{l_{\xi^\alpha_i}\}$
by (*),
the supports of the elements $g_\alpha|J$ and $g_\beta|J$ can intersect
only on the sets 
$S_{\xi^\alpha_i}\cap S_{\xi^\beta_i}\cap J_i$ for
$i<j$ where $g_\alpha$ and $g_\beta$ have the same values $a_i$,
which gives that $\|g_\alpha-g_\beta\|\leq\max(\|g_\alpha\|, \|g_\beta\|)\leq 1$ and consequently
$\|f_\alpha-f_\beta\|< 1+\varepsilon$ as required.

To obtain (**) first note that since $S_\alpha$s are
countable and we have uncountably many
distinct $l_{\xi^\alpha_i}$s, by passing to an uncountable subset we may assume that
$l_{\xi^\alpha_i}\not\in S_{\xi^\beta_i}$ for any $\beta<\alpha<\omega_1$ and any $i<j$.
Now find distinct $\beta_n<\omega_1$ such that 
for every $i<j$ the sequence $l_{\xi^{\beta_n}_i}$ converges to $r_i\in \R$
and find $\sup\{\beta_n: n\in \N\}<\alpha<\omega_1$ such that for every $i<j$ we have
$l_{\xi^{\alpha}_i}\not=r_{i}$. Then 
there is $n\in \N$ such that $l_{\xi^{\beta_n}_i}\not\in S_{\xi^\alpha_i}$ for each $i<j$,
as required for (**).

To see that $S_{C_0(K)}$ admits an uncountable $1$-equilateral set one can use 
the proof of Proposition \ref{C0K-total} but we can also  see that
the set $\{1_{S_\alpha}: \alpha<2^\omega\}$  is $1$-equilateral.

\end{proof}
 
 The following is the locally compact version  of the result
on the behavior of $C(K)$ spaces for compact $K$ from Section 5 of \cite{equi}
(our Theorem \ref{thm-merc} (3)). Compared to it we need to add
one alternative behavior in item (2) below. This alternative already takes place in {\sf ZFC} for
some separable locally compact $K$ (not only for $c_0(\omega_1)$ corresponding to a nonseparable discrete space)
 for this see Proposition \ref{C0K-strange}. The following result also requires a hypothesis
on the weight of $K$ (the first alternative does not require such hypothesis as
the existence of an uncountable $2$-equilateral set is equivalent to the existence of such a
set in a subspace of density $\omega_1<2^\omega$, under the negation of {\sf CH}).
This additional hypothesis on the weight is necessary as seen in Proposition \ref{C0K-counter}.

\begin{proposition}\label{C0K-small} Assume {\sf MA} and
the negation of {\sf CH}. Let $K$ be a locally compact space of weight less than $2^\omega$.
Then either 
\begin{enumerate}
\item $S_{C_0(K)}$ admits an uncountable $2$-equilateral  set, hence the sphere $S_{C_0(K)}$
is dichotomous with ${\mathsf K}^+(S_{C_0(K)})=[0, 2)$ and
${\Sigma}(S_{C_0(K)})=\{2\}$ or else 
\item for every $\varepsilon>0$ the sphere
$S_{C_0(K)}$ is the union of countably many sets of diameters not bigger than $1+\varepsilon$.
Hence the sphere $S_{C_0(K)}$
is dichotomous with ${\mathsf K}^+(S_{C_0(K)})=[0, 1]$ and
${\Sigma}(S_{C_0(K)})=(1,2]$; $S_{C_0(K)}$ admits
an uncountable $1$-equilateral set but no $r$-equilateral set for $r>1$.
\end{enumerate}
\end{proposition}
\begin{proof}  Fix $K$ as in the theorem.  Let $\kappa$ be the weight of $K$. Elements of an
open basis of $K$ of cardinality $\kappa$ having compact closures still form a basis for
$K$. So using the Tietze theorem for locally compact spaces we can find
a subfamily of $C_0(K)$ of cardinality $\kappa$ which separates points and
vanishes nowhere. It follows from Stone-Weierstrass theorem for locally compact
spaces that the algebra $\mathcal G$ generated in $C_0(K)$ 
by this family is norm dense in $C_0(K)$ (and of cardinality $\kappa$). Let $\F\subseteq S_{C_0(K)}$ be
the family of  the normalizations  of all elements of $\mathcal G$. It follows
that $\F$ is norm dense in $S_{C_0(K)}$ and of cardinality $\kappa$. 

Let $\B$ be a family of open sets with compact closures  of $K$ which is closed with respect 
to  taking finite unions and
contains all sets of the form $\{x\in K: f(x)<q\}$ for rationals $q<0$ and $f\in \F$
and all sets of the form $\{x\in K: f(x)>q\}$ for rationals $q>0$ and $f\in \F$. It is clear that
 $\B$ is of cardinality $\kappa$
less than $2^\omega$. 

 Consider the partial order
$$\PP=\{(U, V): U, V\in \B, \ \overline U\cap \overline V=\emptyset\}$$ with the order
$(U', V')\leq (U, V)$ if and only if $U'\supseteq U$ and $V'\supseteq V$.
Note that $(U, V)$ is incompatible with $(U', V')$ if and only
if $\overline U\cap \overline V'\not=\emptyset$ or $\overline U'\cap \overline V\not=\emptyset$
as otherwise $(U\cup U', V\cup V')\leq (U, V), (U', V')$ since $\B$ is closed under taking finite unions.

For any $p=(U, V)\in \PP$ we can define a function $f\in C_0(K)$ satisfying $f|\overline U=-1$, $f|\overline V=1$
and $\|f\|_\infty=1$.  Note that, by the above observation about the incompatibility in $\PP$,
 an uncountable antichain $\A\subseteq \PP$ give rise
to uncountable $\{f_p: p\in \A\}$ which is $2$-equilateral.

So either (1) of the proposition holds or  $\PP$ satisfies the c.c.c. Under {\sf MA} this means that
$\PP$ is $\sigma$-centered by the hypothesis on the cardinality of $\PP$ (Theorem \ref{ma}).
So let $\PP=\bigcup_{n\in \N}{\PP_n}$, where each $\PP_n$ is centered.
For every $f\in \F$ for every rational $\varepsilon>0$ consider
$$p_f^\varepsilon=(\{x\in K: f(x)<-\varepsilon\}, \{x\in K: f(x)>\varepsilon\})\in \PP.$$
Note that if $p_f^\varepsilon, p_g^\varepsilon\in \PP_n$ 
for some $n\in \N$, then $f(x)>\varepsilon$ implies
$g(x)\geq -\varepsilon$ and $g(x)>\varepsilon$ implies
$f(x)\geq -\varepsilon$ for any $x\in K$, so
$\|f-g\|\leq\max(1+\varepsilon, 2\varepsilon)$. So if $2\varepsilon\leq 1+\varepsilon$,
we obtain
that   $\F\subseteq S_{C_0(K)}$ is the union of countably many sets 
$$\X_n=\{f\in \F: p_f^\varepsilon\in \PP_n\}$$
for $n\in \N$ of diameter at most $1+\varepsilon$. 
Since $\varepsilon>0$ is an arbitrary rational and
$\F$ is dense in $S_{C_0(K)}$, we conclude that
for every $\varepsilon>0$ the unit sphere $S_{C_0(K)}$
is the countable union of sets of diameters less that $1+\varepsilon$, so (2) of the proposition holds.
\end{proof}

\section{concluding remarks}

Many natural questions besides $(S')$ and $(E')$ from the introduction remain open.
For example little is known on the existence of uncountable equilateral  sets in Banach
spaces of densities between $2^\omega$ and $2^{2^\omega}$. Also
WCG classes of spaces considered in Section 4 are very particular and it is unclear
to what extent they represent the general behavior of WLD Banach spaces. 
It would be interesting to have some general Ramsey-type principles
for these classes of Banach spaces like the principles obtained from {\sf OCA}
and {\sf MA} with the negation of {\sf CH} for subspaces of $\ell_\infty$
and spaces of continuous functions. One should note here
that the unit spheres of nonseparable reflexive Banach spaces admit uncountable
$(1+\varepsilon)$-separated sets by a result of \cite{hkr} but at least
under {\sf CH} there is a hilbertian nonseparable Banach space
which does not admit an uncountable equilateral set by a result of \cite{pk-kamil}.

We are also quite far from having
a list of all possible sets ${\mathsf K}^+(S_\X)$ and ${\Sigma}(S_\X)$
for Banach spaces $\X$ (depending on properties of $\X$). Only 
the case of spaces of continuous functions is complete under an
additional set-theoretic hypothesis (Theorem \ref{thm-merc}(3) and Proposition \ref{C0K-small}).
The notions of ${\mathsf K}^+(S_\X)$, $\Sigma(S_\X)$  
clearly do not capture all potentially possible relevant phenomena.
For example we do not know if it is possible that there is 
a Banach space $\X$ with the unit sphere
$S_\X$ where there is an uncountable $1$-separated set (or even
an uncountable $1$-equilateral set)
but no uncountable $(1+)$-separated set. We even do not know
if there is a unit sphere $S_\X$ with ${\mathsf k}^+(S_\X)=1$
which is almost dichotomous but not dichotomous.
Another topic not touched in this paper concerns the uncountable cardinalities
of separated or equilateral sets depending in the density of the space (see e.g. \cite{cuth}).

\bibliographystyle{amsplain}

\end{document}